\documentclass{amsart}
\usepackage{amsfonts,amsmath,amssymb,color}

\newcommand{\bydef}{:=}
\newcommand{\defby}{=:}
\newcommand{\bi}{\boldsymbol{i}}

\newcommand{\wh}[1]{\widehat{#1}}
\newcommand{\wt}[1]{\widetilde{#1}}
\newcommand{\wb}[1]{\overline{#1}}

\newcommand{\ve}{\varepsilon}
\newcommand{\vp}{\varphi}
\newcommand{\ld}{\ldots}

\newcommand{\diag}{\mathrm{diag}}
\newcommand{\sgn}{\mathrm{sgn}}
\DeclareMathOperator*{\ot}{\otimes}

\newcommand{\id}{\mathrm{id}}

\newcommand{\cA}{\mathcal{A}}
\newcommand{\cB}{{\mathcal B}}
\newcommand{\cC}{\mathcal{C}}
\newcommand{\cD}{\mathcal{D}}
\newcommand{\cE}{\mathcal{E}}

\newcommand{\cQ}{\mathcal{Q}}

\newcommand{\cU}{\mathcal{U}}
\newcommand{\cV}{\mathcal{V}}
\newcommand{\cY}{\mathcal{Y}}

\newcommand{\NN}{\mathbb{N}}
\newcommand{\ZZ}{\mathbb{Z}}

\newcommand{\QQ}{\mathbb{Q}}

\newcommand{\RR}{\mathbb{R}}
\newcommand{\KK}{\mathbb{K}}
\newcommand{\LL}{\mathbb{L}}
\newcommand{\CC}{\mathbb{C}}
\newcommand{\HH}{\mathbb{H}}

\newcommand{\FF}{\mathbb{F}}
\newcommand{\chr}[1]{\mathrm{char}\,#1}

\DeclareMathOperator{\rad}{\mathrm{rad}}
\DeclareMathOperator{\Hom}{\mathrm{Hom}}
\DeclareMathOperator{\End}{\mathrm{End}}

\DeclareMathOperator{\Gal}{\mathrm{Gal}}

\DeclareMathOperator{\Ext}{\mathrm{Ext}}
\DeclareMathOperator{\Tor}{\mathrm{Tor}}

\DeclareMathOperator{\Aut}{\mathrm{Aut}}
\DeclareMathOperator{\Int}{\mathrm{Int}}

\DeclareMathOperator{\supp}{\mathrm{Supp}\,}
\DeclareMathOperator{\Supp}{\mathrm{Supp}}

\newtheorem{theorem}{Theorem}[section]

\newtheorem{lemma}[theorem]{Lemma}
\newtheorem{corollary}[theorem]{Corollary}
\newtheorem{proposition}[theorem]{Proposition}

\newtheorem{example}[theorem]{Example}
\newtheorem{remark}[theorem]{Remark}
\newtheorem{definition}[theorem]{Definition}

\begin{document}
\title{Graded-division algebras over arbitrary fields}
\author{Yuri Bahturin}
\address{Department of Mathematics and Statistics, Memorial University of Newfoundland, St. John's, NL, A1C5S7, Canada}
\email{bahturin@mun.ca}
\author{Alberto Elduque}
\address{Departamento de Matem\'{a}ticas e Instituto Universitario de Matem\'aticas y Aplicaciones, Universidad de Zaragoza, 50009 Zaragoza, Spain}
\email{elduque@unizar.es}
\author{Mikhail Kochetov}
\address{Department of Mathematics and Statistics, Memorial University of Newfoundland, St. John's, NL, A1C5S7, Canada}
\email{mikhail@mun.ca}

\thanks{The first author acknowledges support by Discovery Grant 227060-14 of the Natural Sciences and Engineering Research Council (NSERC) of Canada. 
The second author was supported by grants MTM2017-83506-C2-1-P (AEI/FEDER, UE) and E22\_17R (Gobierno de Arag\'on, Grupo de referencia ``Algebra y Geometr\'ia'', cofunded by Feder 2014-2020 ``Construyendo Europa desde Arag\'on''). 
The third author acknowledges support by NSERC Discovery Grant 2018-04883.}

\begin{abstract}
A graded-division algebra is an algebra graded by a group such that all nonzero homogeneous elements are invertible. This includes division algebras equipped with an arbitrary group grading (including the trivial grading). 
We show that a classification of finite-dimensional graded-central graded-division algebras over an arbitrary field $\FF$ can be reduced to the following three classifications, for each finite Galois extension $\LL$ of $\FF$: (1) finite-dimensional central division algebras over $\LL$, up to isomorphism; 
(2) twisted group algebras of finite groups over $\LL$, up to graded-isomorphism; (3) $\FF$-forms of certain graded matrix algebras with coefficients in $\Delta\otimes_\LL\cC$ where $\Delta$ is as in (1) and $\cC$ is as in (2).
As an application, we classify, up to graded-isomorphism, the finite-dimensional graded-division algebras over the field of real numbers (or any real closed field) with an abelian grading group. We also discuss group gradings on fields.
\end{abstract}

\maketitle

\section{Introduction}

Let $\cA$ be an algebra over a commutative ring $\FF$ and let $G$ be a group. We will write the operation of $G$ as multiplication and denote its identity element by $e$.
A \emph{$G$-grading} on $\cA$ is an $\FF$-module decomposition $\cA=\bigoplus_{g\in G}\cA_g$ such that $\cA_g\cA_h\subset\cA_{gh}$ for all $g,h\in G$. 
If such a decomposition is fixed, $\cA$ is said to be a \emph{$G$-graded algebra}. The nonzero elements of $\cA_g$ are said to be \emph{homogeneous of degree $g$}. 
An $\FF$-submodule (in particular, a subalgebra or an ideal) $\cU\subset\cA$ is \emph{graded} if $\cU=\bigoplus_{g\in G}\cU_g$ where $\cU_g\bydef \cU\cap\cA_g$. 
In this case, the \emph{support} of $\cU$ is the set $\Supp(\cU)\bydef\{g\in G\mid \cU_g\ne 0\}$. For any subset $H\subset G$, we define a graded submodule $\cA_H\bydef\bigoplus_{g\in H}\cA_g$, with $\Supp(\cA_H)\subset H$.
If $H$ is a subgroup of $G$, then $\cA_H$ is a graded subalgebra of $\cA$.
In this paper, we will deal exclusively with associative algebras over a field $\FF$. 

Many concepts and classical results of ring theory have their $G$-graded analogues. For example, a \emph{homomorphism of $G$-graded algebras} $\cA\to\cB$ is an algebra homomorphism that maps $\cA_g$ to $\cB_g$, for all $g\in G$. In particular, this gives the notion of \emph{isomorphism of graded algebras}, or \emph{graded-isomorphism} for short.
Given a $G$-graded algebra $\cA$, a \emph{graded left $\cA$-module} is a left $\cA$-module $\cV$ with a vector space decomposition $\cV=\bigoplus_{g\in G}\cV_g$ such that $\cA_g\cV_h\subset\cV_{gh}$ for all $g,h\in G$. 
$\cA$ is said to be \emph{graded-simple} if $\cA^2\ne 0$ and $\cA$ has no graded ideals except $0$ and $\cA$. In particular, this holds if $\cA$ is a \emph{graded-division algebra}, i.e., a unital $G$-graded algebra in which every nonzero homogeneous element is invertible. 
It is easy to see that every graded module over a graded-division algebra is free: more precisely, it admits a basis consisting of homogeneous elements.
The importance of graded-division algebras can be illustrated by the following graded analogue of a classical result of Wedderburn: $\cA$ is graded-simple and satisfies the descending chain condition on graded left ideals if and only if there exists a $G$-graded algebra $\cD$ and a graded right $\cD$-module $\cV$ such that $\cD$ is a graded-division algebra, $\cV$ has finite rank, and $\cA\simeq \End_{\cD}(\cV)$ as a $G$-graded algebra (see \cite[Theorem 2.6]{EK}). Here the $G$-grading on $\End_\cD(\cV)$ is defined by declaring an operator $r$ to be homogeneous of degree $g$ if $r(\cV_h)\subset\cV_{gh}$, for all $h\in G$. For a given $\cA$, the pair $(\cD,\cV)$ is unique up to isomorphism and --- a new feature that appears in the graded setting --- shift of grading: $\cV$ may be replaced by $\cV^{[g]}$, for any $g\in G$, where $\cV^{[g]}$ is $\cV$ as a vector space, but with each $\cV_h$ now declared to be the homogeneous component of degree $hg$ and the right $\cD$-module structure of $\cV$ now making $\cV^{[g]}$ a graded right ${}^{[g^{-1}]}\cD^{[g]}$-module (see \cite[Section 2.1]{EK}).

Graded-division algebras were studied in \cite{Kar} in terms of group extensions. Indeed, if $\cD$ is a graded-division algebra, then $T\bydef \Supp(\cD)$ is a subgroup of $G$, $\cD_e$ is a division algebra, and we have the following short exact sequence of groups:
\begin{equation}\label{eq:short_exact}
1\to\cD_e^\times\to\cD_{\mathrm{gr}}^\times\to T\to 1,
\end{equation}
where $\cD_e^\times$ is the group of nonzero elements of $\cD_e$, $\cD_{\mathrm{gr}}^\times$ is the group of nonzero homogeneous elements of $\cD$, the map $\cD_e^\times\to\cD_{\mathrm{gr}}^\times$ is the inclusion, and the map $\cD_{\mathrm{gr}}^\times\to T$ sends a homogeneous element to its degree. Thus, we can assign to each graded-division algebra with support $T$ a group extension of the form 
\[
1\to\Delta^\times\to E\to T\to 1,
\] 
where $\Delta$ is a division algebra and the image of the map $T\to\mathrm{Out}(\Delta^\times)$ is contained in the subgroup $\mathrm{Out}(\Delta)$. (Note that the group $\Delta^\times$ and the algebra $\Delta$ have the same inner automorphisms, also $Z(\Delta^\times)=Z(\Delta)^\times$.)
It is shown in \cite{Kar} that the above assignment is an equivalence of categories if we take as morphisms, on the  one hand, homomorphisms of graded algebras and, one the other hand, homomorphisms of group extensions $(\alpha,\beta,\gamma)$ where $\gamma=\id_T$ and $\alpha$ is the restriction of an algebra map. In principle, this yields a classification of graded-division algebras $\cD$ with $\Supp(\cD)=T$ and $\cD_e\simeq \Delta$ in terms of homomorphisms $\rho:T\to\mathrm{Out}(\Delta)$ and group cohomology: $\rho$ must be such that the corresponding obstruction in $H^3(T,Z(\Delta^\times))$ vanishes, and then the isomorphism classes of $\cD$ with a fixed $\rho$ (up to conjugation in $\mathrm{Out}(\Delta)$) are in bijection with the orbits in $H^2(T,Z(\Delta^\times))$ under a certain twisted action of $\Aut(\Delta,\rho)$ (see \cite{Kar} for details). It should be noted, however, that even if the computation of the relevant cohomology groups and orbits therein can be made, this still does not give us an easy way to construct the corresponding graded-division algebras and study their properties (for example, determine which of them are simple as algebras). Graded-division algebras graded by the group $\ZZ$ were considered in \cite{GO}.

In the special case where $\FF$ is algebraically closed and $\cD$ is finite-dimensional, the situation simplifies: $|T|<\infty$ and $\cD_e=\FF$, so $\cD$ is graded-isomorphic to the twisted group algebra $\FF^\tau T$ for some $2$-cocycle $\tau:T\times T\to \FF^\times$ (with $T$ acting trivially on $\FF^\times$), and $\FF^{\tau} T$ is graded-isomorphic to $\FF^{\tau'} T$ if and only if $[\tau]=[\tau']$ in $H^2(T,\FF^\times)$. 

In this paper we want to study finite-dimensional graded-division algebras over an arbitrary field $\FF$ and show how the general case can be reduced to the case of a twisted group algebra, i.e., $\cD_e=\FF$, by means of Galois descent (see Section \ref{se:gen}). This latter case has been studied, for example, in relation to projective representations of finite groups. Here we will limit our consideration of this case to abelian groups, for which all twisted group algebras can be explicitly classified up to isomorphism of graded algebras (Section \ref{se:1d}). As an application, we give an explicit classification up to isomorphism of the finite-dimensional graded-division algebras with abelian support in the case $\FF=\RR$ (Section \ref{se:real}). This classification was previously known assuming $\cD$ is simple as an algebra \cite{R}. One can also classify graded-division algebras up to \emph{equivalence}, i.e., an isomorphism of algebras $\cD\to\cD'$ that maps $\cD_t$ to $\cD'_{\gamma(t)}$ where $\gamma:T\to T'$ is a group isomorphism. In the case $\FF=\RR$, such a classification of finite-dimensional graded-simple algebras was obtained in \cite{BZ18} (see also \cite{BZ16,R} for the case of $\cD$ simple as an algebra). It should be noted that, for an abelian group $G$, the (cocycle twisted) loop algebra construction in \cite{Eld} (see also \cite{ABFP,BK,BSZ}) can be used to reduce the classification of $G$-graded algebras that are graded-central and graded-simple (resp., graded-division) to that of gradings by quotients of $G$ on central simple (resp., division) algebras.
 
Another interesting question is the following: what division algebras over $\FF$ admit nontrivial gradings and what are the possible supports? Clearly, any grading on a division algebra yields a graded-division algebra, but in general it is difficult to determine which graded-division algebras are in fact division algebras. In Section~\ref{se:fields}, we will discuss this question in the commutative setting, i.e., gradings on field extensions of $\FF$.

\section{General case}\label{se:gen}

In this section we discuss general finite-dimensional graded-division algebras over an arbitrary field and suggest a method to reduce their classification to some more particular problems. Let 
\begin{equation}\label{GSe1}
 \cD=\bigoplus_{t\in T}\cD_t
\end{equation}
 be such a graded algebra over a field $\FF$, where $T\bydef\supp \cD$ is a finite group. Note that, unless $T$ is abelian, the center $Z(\cD)$ is not necessarily a graded subalgebra of $\cD$, but we can still define $Z(\cD)_e\bydef Z(\cD)\cap\cD_e$. It is a subfield of the finite-dimensional division algebra $\cD_e$ and it contains $\FF$. If $Z(\cD)_e=\FF$ then $\cD$ is said to be \emph{graded-central}. In any case, $\cD$ is an algebra over $Z(\cD)_e$ and the decomposition \eqref{GSe1} is a grading of $\cD$ as a $Z(\cD)_e$-algebra. Therefore, it is natural to view $Z(\cD)_e$ as the ground field. From now on, we assume that $\cD$ is a graded-central. We will also fix nonzero elements $X_t\in\cD_t$ for all $t\in T$. Then $\cD_t=\cD_e X_t=X_t\cD_e$ (so $\cD$ can be regarded as a crossed product). 

Now let us consider the center $\LL=Z(\cD_e)$ of the division algebra $\cD_e$. This is a finite field extension of $\FF$. In the following lemma we will show that $\LL$ is always  a Galois extension of $\FF$ and determine its Galois group. For this, we will need the centralizer of $\LL$ in $\cD$:
\[
\mathrm{Cent}_\cD(\LL)\bydef\{ x\in \cD\mid xy=yx\text{ for all }y\in \LL\}.
\]  
Since $\LL\subset\cD_e$, $\mathrm{Cent}_\cD(\LL)$ is a graded subalgebra of $\cD$. We denote by $K$ the support of this subalgebra. Since the elements of $\cD_e$ centralize $\LL$, we have
\[
\mathrm{Cent}_\cD(\LL)=\bigoplus_{t\in K}\cD_t\defby\cD_K.
\]

\begin{lemma}\label{GSl1} 
The center $\LL=Z(\cD_e)$ of the identity component of $\cD$ is a Galois extension of the base field $\FF$, and its Galois group $\Gal(\LL/\FF)$ is isomorphic to $T/K$ where $K=\Supp\mathrm{Cent}_\cD(\LL)$ is a normal subgroup of $T$. 
\end{lemma}

\begin{proof} 
For any $t\in T$, the inner automorphism $\Int X_t$ of $\cD$, $y\mapsto X_tyX_t^{-1}$, maps $\cD_e$ onto itself, so induces an automorphism of $\cD_e$. Being the center of $\cD_e$, $\LL$ is stable under all automorphisms of $\cD_e$. As a result, $\Int X_t$ restricts to an automorphism $(\Int X_t)|_\LL$ of $\LL$ over $\FF$. Clearly, $(\Int X_t)|_\LL$ does not depend on the choice of $X_t$ in $\cD_t$. Therefore, the map $t\mapsto (\Int X_t)|_\LL$ is a well-defined homomorphism $T\to \Aut_\FF(\LL)$. Now $(\Int X_t)|_\LL=\id_\LL$ if and only if $\cD_t$ centralizes $\LL$, so the kernel of this homomorphism is $K$. Also, the fixed points in $\LL$ under this action of $T$ are precisely the elements of $\FF=Z(\cD)_e$. By Artin's theorem \cite[Theorem 2 in Section VIII.1]{L}, $\LL$ is a Galois extension of $\FF$ and $T/K$ is isomorphic to $\Gal (\LL/\FF)$.
\end{proof}

In what follows, we are primarily interested in the case $\LL=Z(\cD_e)$, but our results hold in a slightly more general setting: $\LL$ can be any Galois extension of $\FF$ contained in $Z(\cD_e)$. It follows from Lemma \ref{GSl1} that the Galois group $\Gal(\LL/\FF)$ is still isomorphic to $T/K$ where $K=\Supp\mathrm{Cent}_\cD(\LL)$. 

Now let us extend the scalars from $\FF$ to $\LL$, i.e., consider the $\LL$-algebra $\wt{\cD}=\cD\ot_\FF \LL$ with the natural grading $\wt{\cD}=\bigoplus_{t\in T}\wt{\cD}_t$ where $\wt{\cD}_t\bydef \cD_t\otimes_\FF \LL$. 

It is clear that the property of being graded-central is invariant under scalar extensions. Indeed, let $\cA$ be a unital $G$-graded algebra over a field $\FF$ and let $\KK$ be any field extension of $\FF$. Then $Z(\cA\ot_\FF \KK)=Z(\cA)\ot_\FF \KK$ and, hence, $Z(\cA\ot_\FF \KK)_e=Z(\cA)_e\ot_\FF \KK$, which shows that $\cA$ is graded-central if and only if so is $\cA\ot_\FF \KK$. (This is valid without assuming $\cA$ unital or associative if we replace the center by centroid.) As is well known from the case $G=\{e\}$, being graded-simple or graded-division is not preserved under scalar extensions in general. Following \cite{ABFP}, we will consider the property of being graded-central and graded-simple at the same time, which we refer to as \emph{graded-central-simple}.

\begin{lemma}
The property of being graded-central-simple is invariant under scalar extensions.
\end{lemma}

\begin{proof}
Let $\cA$ be a unital $G$-graded algebra over a field $\FF$ and let $\KK$ be any field extension of $\FF$. We can consider $\cA$ as an algebra with one binary operation (multiplication) and a family of unary operations $\{p_g\mid g\in G\}$ where $p_g:\cA\to\cA$ sends each element of $\cA$ to its component of degree $g$. Then $Z(\cA)_e$ can be interpreted as the centroid of this expanded algebra structure. Moreover, the ideals of this structure are precisely the graded ideals of $\cA$, and its scalar extension from $\FF$ to $\KK$ gives precisely the algebra  $\cA\ot_\FF \KK$ with its natural grading. The result follows.
\end{proof}

Therefore, our graded algebra $\wt{\cD}$ is graded-central-simple. However, it is not a graded-division algebra unless $\LL=\FF$. Indeed, write $\Gal(\LL/\FF)=\{\sigma_1,\sigma_2\ld,\sigma_k\}$, where $k=|\LL:\FF]=[T:K]$ and $\sigma_1=\id$. It is well known that the mapping $\lambda\ot\mu\mapsto (\sigma_1(\lambda)\mu,\ld,\sigma_k(\lambda)\mu)$, for all $\lambda,\mu\in \LL$,  extends to an isomorphism 
\[
\LL\ot_\FF \LL\,
\stackrel{\sim}{\longrightarrow}\,\underbrace{\LL_1\times \LL_2\times\cdots\times \LL_k}_{k\text{ \tiny{copies of} }\LL}. 
\]
Consequently,  
\[
\wt{\cD}_e=\cD_e\ot_\FF \LL=(\cD_e\ot_\LL)\ot_\FF \LL=\cD_e\ot_\LL(\LL\ot_\FF \LL)\simeq \Delta_1\times\Delta_2\times\cdots\times\Delta_k,
\]
where $\Delta_i\bydef\cD_e\ot_\LL \LL_i$ are the \emph{Galois twists} of $\cD_e$: $\Delta_i$ is isomorphic to $\cD_e$ as an $\FF$-algebra but its $\LL$-algebra structure is twisted in the sense that, in $\Delta_i$, we have $d\cdot\lambda=d\sigma_i^{-1}(\lambda)$ for $d\in\cD_e$ and $\lambda\in \LL$. With our convention $\sigma_1=\id$, we may identify $\Delta_1$ with $\cD_e$ as an $\LL$-algebra. 

Note that $\Gal(\LL/\FF)$ acts on $\wt{\cD}$ in the usual way: $\sigma(d\ot\lambda)=d\ot\sigma(\lambda)$, for all $d\in\cD$ and $\lambda\in \LL$, and that $\Gal(\LL/\FF)$ permutes the primitive (central) idempotents $\ve_1,\ve_2,\ld,\ve_k$ of $\wt{\cD}_e$ simply transitively: $\sigma_j(\ve_1)=\ve_j$, so we may relabel them as $\ve_\sigma=\ve_j$ if $\sigma=\sigma_j$. With this relabeling, we have $\sigma(\ve_\tau)=\ve_{\sigma\tau}$.

The same argument works for $\textrm{Cent}_\cD(\LL)=\cD_K$, since it is an $\LL$-algebra. In particular, $\wt{\cD}_K=\cD_K\otimes_\FF \LL$ has the property that $\ve_1\wt{\cD}_K$  is isomorphic to $\cD_K$ as a graded $\LL$-algebra.

Let us pick $t_j\in T$ so that $t_jK\mapsto \sigma_j$ under the isomorphism $T/K\stackrel{\sim}{\longrightarrow}\Gal(\LL/\FF)$. Then $\{t_1,t_2,\ld,t_k\}$ is a transversal for $K$ in $T$. It is convenient to take $t_1=e$. Since $\Int(X_{t_j}\ot 1)$ acts as $\sigma_j\ot\id$ on $\LL\ot_\FF \LL$, it permutes the idempotents $\{\ve_1,\ve_2,\ld,\ve_k\}$, with the above relabeling, in this way: $\Int(X_{t_j}\ot 1)(\ve_\tau)=\ve_{\tau\sigma_j^{-1}}$. It follows that, for any choice of $t_j$, we have
\begin{equation}\label{eq:eXe}
\ve_i(X_{t_j}^{-1}\ot 1)\ve_1=
\begin{cases}
\ve_i(X_{t_i}^{-1}\ot 1)=(X_{t_i}^{-1}\ot 1)\ve_1 & \text{if }i=j;\\
0 & \text{if }i\ne j.
\end{cases}
\end{equation}

Now, since $\wt{\cD}$ is a finite-dimensional graded-simple $\LL$-algebra, there exists a finite-dimensional graded-division $\LL$-algebra $\cE$ and a graded right $\cE$-module $\cV$ of finite rank such that $\wt{\cD}\simeq\End_{\cE}(\cV)$. 
We can choose for $\cV$ any minimal graded left ideal of $\wt{\cD}$. Since $\ve_1$ is a primitive idempotent of  $\wt{\cD}_e$, we will take  $\cV=\wt{\cD}\ve_1$. With this choice of $\cV$, and taking into account \eqref{eq:eXe}, we will have 
\[
\cE\simeq\End_{\wt{\cD}}(\wt{\cD}\ve_1)=\ve_1\wt{\cD}\ve_1=\ve_1\wt{\cD}_K\ve_1=\ve_1\wt{\cD}_K\simeq\cD_K
\]
as graded $\LL$-algebras. In particular, $\cE_e\simeq\cD_e$ and $\Supp \cE=K$. 

At the same time, 
\begin{equation}\label{eq:decompV}
\cV=\wt{\cD}\ve_1=\ve_1\wt{\cD}\ve_1\oplus\ve_2\wt{\cD}\ve_1\oplus\cdots\oplus\ve_k\wt{\cD}\ve_1,
\end{equation}
as a graded right module over $\ve_1\wt{\cD}\ve_1=\ve_1\wt{\cD}_K$. It follows from \eqref{eq:eXe} that $\ve_i\wt{\cD}\ve_1= (X_{t_i}^{-1}\ot 1)\ve_1\wt{\cD}_K$ and so the set
\[
\{(X_{t_1}^{-1}\ot 1)\ve_1, (X_{t_2}^{-1}\ot 1)\ve_1,\ld,(X_{t_k}^{-1}\ot 1)\ve_1\}
\]
is a basis of $\cV$ consisting of homogeneous elements of degrees $t_1^{-1},t_2^{-1},\ld,t_k^{-1}$, respectively. In the language of \cite{EK}, the multiset in $T/K$ defining the graded right $\cE$-module $\cV$ is $T/K$ with multiplicity $1$ at each point.

Finally, in the case $\LL=Z(\cD_e)$, all $\LL$-linear automorphisms of $\cE_e\simeq\cD_e$ are inner, so we may multiply $X_t$, for each $t\in K$, by a suitable element of $\cD_e$ so that $X_t\in\mathrm{Cent}_\cD(\cD_e)$. Then $\cE\simeq\cD_K\simeq\cD_e\ot_\LL\cC$ as a graded $\LL$-algebra, where 
\[
\cC\bydef\mathrm{Cent}_{\cD}(\cD_e)=\mathrm{Cent}_{\cD_K}(\cD_e)=\bigoplus_{t\in K}\LL X_t.
\]

To summarize:

\begin{theorem}\label{GSt1}
Let $\cD$ be a graded-division algebra (or ring), considered as a $T$-graded algebra over the field $\FF=Z(\cD)_e$, where $T=\Supp\cD$. Assume that $\cD$ is finite-dimensional. Let $\LL$ be a Galois extension of $\FF$ contained in $Z(\cD_e)$ and let $K=\Supp\mathrm{Cent}_\cD(\LL)$. Then $\cD\ot_\FF \LL\simeq\End_{\cE}(\cV)$ as a graded $\LL$-algebra, where $\cE\simeq\cD_K$ is a graded-division algebra over $\LL=Z(\cE)_e$ and $\cV$ is the direct sum of graded right $\cE$-modules of rank $1$, with exactly one from each isomorphism class of such modules. Moreover, in the case $\LL=Z(\cD_e)$ (which is a Galois extension of $\FF$ by Lemma \ref{GSl1}), we have $\cE=\cD_e\ot_\LL\cC$ where $\cC=\mathrm{Cent}_{\cD}(\cD_e)$ is a graded-divison algebra over $\LL$ with $1$-dimensional homogeneous components. \qed
\end{theorem}

The algebra $\wt{\cD}=\cD\ot_\FF \LL$ comes with the canonical Galois descent datum: for every $\sigma\in\Gal(\LL/\FF)$, we have the automorphism $\psi_\sigma$ of the graded $\FF$-algebra $\wt{\cD}$ defined by $\psi_\sigma(d\ot\lambda)=d\ot\sigma(\lambda)$, for all $d\in\cD$ and $\lambda\in \LL$, and we can recover $\cD$ as the set of points in $\wt{\cD}$ fixed by all these automorphisms. Recall that, in general, a \emph{Galois descent datum} for a $G$-graded $\LL$-algebra $\wt{\cA}$ is a homomorphism, $\sigma\mapsto\psi_\sigma$, from $\Gal(\LL/\FF)$ to the group $\Aut_\FF^G(\wt{\cA})$ of automorpisms of $\wt{\cA}$ as a graded $\FF$-algebra such that $\psi_\sigma(x\lambda)=\psi_\sigma(x)\sigma(\lambda)$, for all $x\in\wt{\cA}$ and $\lambda\in \LL$. Given such a datum, the set of fixed points $\cA$ is a graded $\FF$-subalgebra of $\wt{\cA}$, and we have an isomorphism $\cA\ot_\FF \LL\to\wt{\cA}$ sending $a\otimes\lambda\mapsto a\lambda$, i.e., $\cA$ is an \emph{$\FF$-form} of $\wt{\cA}$.

The automorphisms of $\End_{\cE}(\cV)$, where $\cE$ is a graded-division algebra and $\cV$ is a graded right $\cE$-module of finite rank, are described by \cite[Theorem 2.10]{EK}: each $\psi_\sigma$ is given by a pair $(\psi_\sigma^0,\psi_\sigma^1)$, where $\psi_\sigma^0:{}^{[t_\sigma]}\cE^{[t_\sigma^{-1}]}\to\cE$, for some $t_\sigma\in T$, is an isomorphism of graded $\FF$-algebras and $\psi_\sigma^1:\cV^{[t_\sigma^{-1}]}\to\cV$ is a $\psi_\sigma^0$-semilinear isomorphism of graded modules, in the following way: $\psi_\sigma(r)=\psi_\sigma^1r(\psi_\sigma^1)^{-1}$ for all $r\in\End_{\cE}(\cV)$. The pair $(\psi_\sigma^0,\psi_\sigma^1)$ is not quite unique: for any nonzero homogeneous element $d\in\cE$, we can replace $\psi_\sigma^1$ by the map $v\mapsto\psi_\sigma^1(v)d$ and, simultaneously, $\psi_\sigma^0$ by $(\Int d)^{-1}\psi_\sigma^0$. Thus, only the coset $t_\sigma^{-1}(\Supp\cE)$ is determined by $\psi_\sigma$. It is convenient to fix $\psi_{\id}^0=\id_\cE$ and $\psi_\id^1=\id_\cV$.

With the usual identification of $Z(\cE)$ and $Z(\End_\cE(\cV))$, we have $\psi_\sigma\,|_{Z(\cE)}=\psi_\sigma^0\,|_{Z(\cE)}$, so $\psi_\sigma^0:{}^{[t_\sigma]}\cE^{[t_\sigma^{-1}]}\to\cE$ is a $\sigma$-semilinear isomorphism of graded $\LL$-algebras. 

In the setting of Theorem \ref{GSt1}, recall the labeling of idempotents $\ve_1,\ve_2,\ld,\ve_k$ of $\wt{\cD}=\cD\ot_\FF \LL$ as $\ve_\sigma:=\psi_\sigma(\ve_1)$, where $\psi_\sigma$ is the canonical descent datum. Thus, we have $\psi_\sigma(\ve_\tau)=\ve_{\sigma\tau}$. Note that, under the isomorphism $\wt{\cD}\simeq\End_\cE(\cV)$, the element $\ve_\sigma$ corresponds to the projection of $\cV=\wt{\cD}\ve_1$ onto the summand $\cV_\sigma\bydef\ve_\sigma\wt{\cD}\ve_1$ in the direct sum decomposition $\cV=\bigoplus_{\sigma\in\Gal(\LL/\FF)}\cV_\sigma$ (see Equation \eqref{eq:decompV}). Recall that the support of $\cV_\sigma$ is the coset of $K$ in $T$ corresponding to $\sigma^{-1}$ under the isomorphism $T/K\stackrel{\sim}{\longrightarrow}\Gal(\LL/\FF)$ sending $tK\mapsto(\Int X_t)|_\LL$. Since $\psi_\sigma^1$ maps the homogeneous component $\cV_t$ onto $\cV_{tt_\sigma^{-1}}$, for any $t\in T$, it maps $\cV_{t^{-1}K}$ onto $\cV_{(t_\sigma t)^{-1}K}$. It follows that $t_\sigma K$ must be the element of $T/K$ corresponding to $\sigma$. In other words, $\psi_\sigma^1(\cV_\tau)=\cV_{\sigma\tau}$ for all $\sigma,\tau\in\Gal(\LL/\FF)$.

We will now prove the converse:

\begin{theorem}\label{GSt2}
Let $\cE$ be a finite-dimensional graded-division algebra with support $K$ over a field $\LL$, with $Z(\cE)_e=\LL$. 
Suppose $\LL$ is a finite Galois extension of $\FF$ and $T$ is a group containing $K$ as a normal subgroup such that $T/K\simeq\Gal(\LL/\FF)$. Consider $\cE$ as a $T$-graded algebra and let $\cV$ be the direct sum of graded right $\cE$-modules of rank $1$, with exactly one from each of the $[T:K]$ isomorphism classes of such modules: $\cV=\bigoplus_{\sigma\in\Gal(\LL/\FF)}\cV_\sigma$, where we chose the labeling in such a way that $\Supp(\cV_\sigma)$ is the coset of $K$ in $T$ corresponding to $\sigma^{-1}$ under the above isomorphism. Suppose that $\wt{\cD}\bydef\End_\cE(\cV)$ has a descent datum $\Gal(\LL/\FF)\to\Aut_\FF^T(\wt{\cD})$, $\sigma\mapsto\psi_\sigma$, where $\psi_\sigma(r)=\psi_\sigma^1r(\psi_\sigma^1)^{-1}$ for all $r\in\wt{\cD}$ and $\psi_\sigma^1(\cV_\tau)=\cV_{\sigma\tau}$ for all $\sigma,\tau\in\Gal(\LL/\FF)$. Then the $\FF$-form $\cD$ of the graded $\LL$-algebra $\wt{\cD}$ determined by the descent (i.e., the set of fixed points in $\wt{\cD}$) is a graded-division algebra with support $T$ and $Z(\cD)_e=\FF$. Moreover, $\LL$ is isomorphic to a subfield of $Z(\cD_e)$ and $\cE\simeq\cD_K$ as graded $\FF$-algebras.
\end{theorem}

\begin{proof}
Pick representatives in each coset of $K$ in $T$ and label them $t_\sigma$, according to the isomorphism $T/K\simeq\Gal(\LL/\FF)$. It is convenient to have $t_\id=e$. Let $v_\sigma\in\cV$ be a nonzero element of degree $t_\sigma^{-1}$. Then $\cV_\sigma=v_\sigma\cE$ and $\{v_\sigma\mid\sigma\in\Gal(\LL/\FF)\}$ is a basis of $\cV$ as an $\cE$-module. Relative to this basis, the homogeneous elements of $\End_\cE(\cV)$ are represented by matrices that have at most one nonzero entry in each row and in each column. Indeed, if $r\in \End_\cE(\cV)$ is homogeneous of degree $g\in T$, then $r(\cV_\tau)\subset\cV_{\tau\sigma_0}$ for all $\tau\in\Gal(\LL/\FF)$, where $\sigma_0$ is the element of $\Gal(\LL/\FF)$ corresponding to $g^{-1}K$. Hence, $r(v_\tau)=v_{\tau\sigma_0}d_\tau$ for some homogeneous $d_\tau\in\cE$. 

Without loss of generality, we may assume that $\psi_\sigma^1$ is a ($\psi_\sigma^0$-semilinear) isomorphism $\cV^{[t_\sigma^{-1}]}\to\cV$. Then $\psi_\sigma^1(v_\tau)= v_{\sigma\tau}c_{\sigma,\tau}$ where $0\ne c_{\sigma,\tau}\in \cE_{t_{\sigma\tau}t_\tau^{-1}t_\sigma^{-1}}$. Since $\psi_\id^1=\id_\cV$, it follows that $c_{\id,\tau}=1$.  

Now, the element $r$, as above, belongs to $\cD$ if and only if $r$ is a fixed point of all $\psi_\sigma$, i.e., $r\psi_\sigma^1=\psi_\sigma^1r$ for all $\sigma\in\Gal(\LL/\FF)$. Evaluating both sides on the basis elements $v_\tau$, we obtain:
\begin{equation*}
r\in\cD\quad\Leftrightarrow\quad d_{\sigma\tau}=c_{\sigma,\tau\sigma_0}\psi_\sigma^0(d_\tau)c_{\sigma,\tau}^{-1}\text{ for all }\sigma,\,\tau\in\Gal(\LL/\FF).
\end{equation*}
In particular, if $r\in\cD$ then $d_\sigma=c_{\sigma,\sigma_0}\psi_\sigma^0(d_\id)c_{\sigma,\id}^{-1}$. As a result, all entries of the matrix representing  $r$ are determined by the single entry $d_\id$. Therefore, if $r$ is nonzero then all $d_\sigma$ are invertible, so $\cD$ is indeed a graded-division algebra. By construction, we have $\wt{\cD}\simeq\cD\ot_\FF \LL$, hence $\cD$ has the same support as $\wt{\cD}$, i.e., $T$. Since $Z(\wt{\cD})$ is identified with $Z(\cE)$ and the property of being graded-central is invariant under field extensions, we conclude that $Z(\cD)_e=\FF$.
 
Observe that if $g\in K$ then $\sigma_0=\id$ and $r$ is represented by a diagonal matrix, in which the entry $d_\tau$ is homogeneous of degree $t_\tau g t_\tau^{-1}$. Let us define $\Phi:\cE\to\wt{\cD}$ by sending, for each $g\in K$, every $d\in\cE_g$ to the (diagonal) operator $\Phi(d)\in\End_\cE(\cV)$ defined by $\Phi(d)(v_\sigma)=v_\sigma((\Int c_{\sigma,\id})\psi_\sigma^0(d))$, for all $\sigma\in\Gal(\LL/\FF)$. Clearly, $\Phi$ is an injective homomorphism of $\FF$-algebras (but not $\LL$-algebras), and it preserves degree since $\psi_\sigma^0:{}^{[t_\sigma]}\cE^{[t_\sigma^{-1}]}\to\cE$ and $c_{\sigma,\id}\in\cE_e$. Also, the fixed point condition in the previous paragraph implies that $\Phi(\cE)\supset\cD_K$. We claim that, actually, $\Phi(\cE)=\cD_K$. Indeed, it is sufficient to prove that the dimensions are equal. We have $\dim_\FF(\cD_K)=\dim_\LL(\wt{\cD}_K)=[T:K]\dim_\LL(\cE)$, since $\wt{\cD}_K$ consists of all operators represented by the diagonal matrices with entries in $\cE$. But $[T:K]=[\LL:\FF]$, so $\dim_\FF(\cD_K)= \dim_\FF(\cE)$, which proves the claim. In particular, $\Phi$ maps $\cE_e$ onto $\cD_e$, hence $\LL$ is isomorphic to a subfield of $Z(\cD_e)$.
\end{proof}
 
\begin{remark}\label{GSr1}
Since $\psi_\sigma^0:{}^{[t_\sigma]}\cE^{[t_\sigma^{-1}]}\to\cE$ is a $\sigma$-semilinear isomorphism (and, for any $t\in K$, ${}^{[t]}\cE^{[t^{-1}]}\simeq\cE$ by means of an inner automorphism), it follows that, in Theorems \ref{GSt1} and \ref{GSt2}, for any $t\in T$, the graded $\LL$-algebra ${}^{[t]}\cE^{[t^{-1}]}$ is isomorphic to the $\sigma$-twist of $\cE$ where $\sigma\in\Gal(\LL/\FF)$ corresponds to $tK$ under the isomorphism $T/K\simeq\Gal(\LL/\FF)$. 
In particular, if $T=K\,\mathrm{Cent}_T(K)$ then $\cE$ is isomorphic to all of its Galois twists.
\end{remark}
 
In view of Theorems \ref{GSt1} and \ref{GSt2}, the  classification of finite-dimensional graded-central graded-division algebras over a field $\FF$, up to isomorphism of graded algebras, reduces to the following three problems:
\begin{description}
\item[\sc CP1] The classification, up to isomorphism, of finite-dimensional central division algebras over each finite Galois extension $\LL$ of $\FF$;
\item[\sc CP2] The classification of graded-division algebras with $1$-dimensional components over $\LL$ (in other words, twisted group algebras of finite groups);
\item[\sc CP3] The classification of the descent data on $\End_\cE(\cV)$ as in Theorem \ref{GSt2} where $\cE=\Delta\ot_\LL\cC$, with $\Delta$ is as in {\sc CP1} and $\cC$ is as in {\sc CP2}.
\end{description}
 
Problem 1 is classical and hard for many important fields $\FF$. Problem 3 also appears to be hard in general (but not for $\FF=\RR$, as we will see in Section \ref{se:real}). For any given field $\LL$ and group $T$, Problem 2 can be approached through  group cohomology: one has to compute $H^2(T,\LL^\times)$ and also find a $2$-cocycle in each cohomology class if one wants to construct the corresponding twisted group algebra explicitly. In the next section we will solve Problem 2 for all finite abelian groups.

\section{Graded-division algebras with 1-dimensional components}\label{se:1d}

Consider a graded-division algebra $\cC$ over a field $\LL$ with $1$-dimensional homogeneous components $\cC_t$, $t\in K$. As before, we can fix nonzero elements $X_t\in\cC_t$, for any $t\in K$. Then, for any $s,t\in K$, we have $X_s X_t=\tau(s,t)X_{st}$ where $\tau(s,t)\in\LL^\times$. The function $\tau:K\times K\to\LL^\times$ is a $2$-cocycle with respect to the trivial action of $K$ on $\LL^\times$. Thus, $\cC$ can be regarded as the twisted group algebra $\LL^\tau K$ with its natural $K$-grading. The graded-isomorphism class of $\cC$ is determined by the cohomology class $[\tau]\in H^2(K,\LL^\times)$. From now on, we will assume that $K$ is a finitely generated abelian group. Then we have a split exact sequence:
\begin{equation*}
1\to\Ext_\ZZ(K,\LL^\times)\to H^2(K,\LL^\times) \xrightarrow{\#} \Hom_\ZZ(K\wedge _\ZZ K,\LL^\times)\to 1,
\end{equation*}
where an element of $\Ext_\ZZ(K,\LL^\times)$ is sent to the corresponding class of symmetric $2$-cocycles $K\times K\to\LL^\times$ and the map $\#$ sends $[\tau]\in H^2(K,\LL^\times)$ to the alternating bicharacter $\beta_\tau(s,t)\bydef\tau(s,t)/\tau(t,s)$. A (non-canonical) splitting of the above sequence can be obtained by writing $K$ as a direct product of cyclic groups: 
\begin{equation}\label{eq:cyclic_decomp}
K=\langle a_1\rangle\times\cdots\times\langle a_m\rangle,
\end{equation}
and sending an alternating bicharacter $\beta\in \Hom_\ZZ(K\wedge _\ZZ K,\LL^\times)$ to the bicharacter (hence $2$-cocycle) $\tau\in \Hom_\ZZ(K\ot _\ZZ K,\LL^\times)$ defined by
\[
\tau(a_i,a_j)=
\begin{cases}
\beta(a_i,a_j)&\text{ if } i<j;\\
1&\text{ if }i\ge j.
\end{cases}
\]
So, in principle, $H^2(K,\LL^\times)$ is known, since $\Ext_\ZZ$ is an additive functor and 
\[
\Ext_\ZZ(\ZZ,\LL^\times)=1\quad\text{and}\quad\Ext_\ZZ(\ZZ_n,\LL^\times)=\LL^\times/(\LL^\times)^{[n]},
\] 
where we are using the following

{\sc Notation}. For any abelian group $A$, we define 
\[
A^{[n]}\bydef\{a^n\mid a\in A\}\quad\text{and}\quad A_{[n]}\bydef\{ a\in A\mid a^n=e\}.
\]
Also, $\Tor(A)$ will denote the periodic (torsion) part of $G$ and $\Tor_p(A)$ the $p$-primary component of $A$ (i.e., the union of $A_{[p^k]}$ over $k\in\NN$), where $p$ is a prime number. It is convenient to define $\Tor_{p'}(A)\bydef\prod_{q\ne p}\Tor_q(A)$. Finally, $o(a)$ will denote the order of an element $a\in A$.

We would like to give an explicit description of the graded-division algebra $\cC$. First, we note that $\beta=\beta_\tau$ describes the commutation relations in $\cC$: 
\[
X_sX_t=\beta(s,t)X_tX_s.
\]
(This is the easiest way to see that $\beta_\tau$ is indeed a bicharacter and it is independent of the choice of the elements $X_t$ or, in other words, of the $2$-cocycle $\tau$ representing a given cohomology class.) Unless $\Ext_\ZZ(K,\LL^\times)=1$ --- as is the case, for example when $\LL$ is algebraically closed, because then $\LL^\times$ is a divisible group, $\beta$ alone does not determine the graded-isomorphism class of $\cC$.

\begin{remark}
$\beta$ determines the graded-isomorphism class of $\cC\otimes_\LL\wb{\LL}$, where $\wb{\LL}$ is the algebraic closure of $\LL$. Thus, if $\LL$ is perfect, $\cC$ can be recovered through Galois descent, but we are going to take a more direct approach.
\end{remark}

Let $F_K=\prod_{t\in\Tor(K)}F_t$, where $F_t\bydef\LL^\times/(\LL^\times)^{[o(t)]}$. For any $t\in\Tor(K)$, the element  $X_t^{o(t)}\in\cC_e^\times=\LL^\times$ depends on the choice of $X_t$, however its image in $F_t$ does not. We denote this image by $\mu(t)$. The function $\mu:\Tor(K)\to\bigcup_{t\in\Tor(K)}F_t$ can be regarded as an element of $F_K$. 

The triple $(K,\beta,\mu)$ determines $\cC$ up to isomorphism of graded algebras. Indeed, if we write $K$ as the direct product of cyclic groups \eqref{eq:cyclic_decomp}, we can present $\cC$ in terms of generators and defining relations as follows. First, we give the free associative algebra $\LL\langle X^\pm_1,\ld, X^\pm_m\rangle$ a $K$-grading by setting $\deg X^+_i\bydef a_i$ and $\deg X^-_i\bydef a_i^{-1}$. (This is the $K$-grading induced from the standard $\ZZ^{2m}$-grading of the free associative algebra of rank $2m$ by a homomorphism $\ZZ^{2m}\to K$.) Second, we consider the two-sided ideal $J$ generated by the homogeneous elements of three types:
\begin{enumerate}
	\item[(i)] $X_i^- X_i^+ - 1$ and $X_i^+ X_i^- - 1$ for all $i$;
	\item[(ii)] $X_iX_j-\beta_{ij}X_jX_i$ for all $i<j$;
	\item[(iii)] $X_i^{o(a_i)}-\mu_i$ for all $i$ such that $o(a_i)<\infty$;
\end{enumerate}
where $\beta_{ij}=\beta(a_i,a_j)$ and $\mu_i\in\LL^\times$ is any representative of $\mu(a_i)\in \LL^\times/(\LL^\times)^{[o(a_i)]}$. Then, for any choice of representatives, we have
\[
\cC\simeq\LL\langle X^\pm_1,\ld,X^\pm_m\rangle/J
\]
as graded algebras. Indeed, for each $i$ with $o(a_i)<\infty$, we can find $\lambda_i\in\LL^\times$ such that $\lambda_i^{o(a_i)}=\mu_i X_{a_i}^{-o(a_i)}$, and we let $\lambda_i=1$ for all $i$ with $o(a_i)=\infty$. Then the surjective homomorphism of graded algebras $\LL\langle X_1^\pm,\ld,X^\pm_m\rangle\to\cC$ sending $X^+_i\mapsto\lambda_i X_{a_i}$ and $X^-_i\mapsto\lambda_i^{-1} X_{a_i}^{-1}$ factors through a homomorphism $\LL\langle X^\pm_1,\ld,X^\pm_m\rangle/J\to\cC$, which must be an isomorphism thanks to the following result. 
 
\begin{proposition}\label{D_beta_mu}
For any scalars $\beta_{ij}\in \LL^\times$ ($i<j$) and $\mu_i\in \LL^\times$ satisfying $\beta_{ij}^{o(a_i)}=1$ if $o(a_i)<\infty$ and $\beta_{ij}^{o(a_j)}=1$ if $o(a_j)<\infty$, let $J$ be the (graded) ideal of $\LL\langle X^\pm_1,\ld,X^\pm_m\rangle$ generated by the elements (i), (ii) and (iii) above. Then 
\[
\cD(K,\beta_{ij},\mu_i)\bydef\LL\langle X^\pm_1,\ld,X^\pm_m\rangle/J
\] 
is a graded-division algebra with support $K$ and $1$-dimensional homogeneous components. 
\end{proposition}
 
\begin{proof}
For any $n\in\ZZ_{\ge 0}$, denote $X_i^n\bydef (X_i^+)^n$ and $X_i^{-n}\bydef(X_i^-)^n$. It is clear that $\cD(K,\beta_{ij},\mu_i)$ is spanned by the images of the monomials $X_1^{n_1}\cdots X_m^{n_m}$, where $n_i\in\ZZ$ and $0\le n_i<o(a_i)$ if $o(a_i)<\infty$. All of these monomials have different degrees in $K$, so the homogeneous components of $\cD(K,\beta_{ij},\mu_i)$ have dimension at most $1$, but one needs to check that the support of the grading is the entire group $K$, because a priori some of these monomials could belong to $J$. 
 
This can be done, for example, by first considering the algebra 
\[
\cQ\bydef \LL\langle X^\pm_1,\ld,X^\pm_m\rangle/J_0
\]
where $J_0$ is the ideal generated by the elements (i) and (ii). It is well known and easy to show (for example, by considering iterated Ore extensions) that $\cQ$ has a basis consisting of (the images of) the monomials $X_1^{n_1}\cdots X_m^{n_m}$, $n_i\in\ZZ$. By our condition on the scalars $\beta_{ij}$, the monomials $X_i^{o(a_i)}$, for all $i$ with $o(a_i)<\infty$, are central in $\cQ$. Let $\cA$ be the subalgebra that they generate in $\cQ$. As an $\cA$-module, $\cQ$ has a basis consisting of the monomials satisfying $0\le n_i<o(a_i)$ for all $i$ with $o(a_i)<\infty$. It follows that these monomials form an $\LL$-basis of $\cQ/\wb{J}$, where $\wb{J}=J/J_0$ is the ideal generated by the (images of) the elements (iii). Clearly, $\cQ/\wb{J}\simeq \cD(K,\beta_{ij},\mu_i)$.
\end{proof}

\begin{remark}
The algebra $\cQ$ in the above proof is known as the \emph{quantum torus} of dimension $m$ with parameters $\beta_{ij}$. So, one might call our algebras $\cD(K,\beta_{ij},\mu_i)$ \emph{quantum quasitori}.
\end{remark}

The conditions on $\beta_{ij}$ in Proposition \ref{D_beta_mu} are necessary and sufficient for the existence of an alternating bicharacter $\beta: K\times K\to \LL^\times$ such that $\beta(a_i,a_j)=\beta_{ij}$ for all $i<j$. Clearly, the scalars $\beta_{ij}$ determine $\beta$ uniquely, and $\beta$ is the bicharacter associated to the graded-division algebra $\cD(K,\beta_{ij},\mu_i)$ by its commutation relations. Also, the function $\mu\in F_K$ associated to $\cD(K,\beta_{ij},\mu_i)$ satisfies $\mu(a_i)=\mu_i(\LL^\times)^{[o(a_i)]}$ for all $i$ with $o(a_i)<\infty$. It follows that $\cD(K,\beta_{ij},\mu_i)\simeq\cD(K,\beta_{ij}^\prime,\mu_i^\prime)$, as graded algebras, if and only $\beta_{ij}=\beta_{ij}^\prime$ for all $i<j$ and $\mu_i\equiv\mu_i^\prime\pmod{(\LL^\times)^{[o(a_i)]}}$ for all $i$ with $o(a_i)<\infty$. 

\begin{remark}\label{GDr1} 
With $\beta$ fixed, we can view $\cC$ as an object in the Yetter-Drinfeld category ${}_K^K\cY\cD$ (see e.g. \cite[Chapter 10]{Mont}), with $\LL K$-coaction given by the grading: $x\mapsto g\ot x$ for all $x\in\cC_g$, and the $\LL K$-action induced from the coaction through $\beta$: $g\cdot x=\beta(g,h)x$ for all $x\in\cC_h$. Then any decomposition of $K$ as a direct product of cyclic groups allows us to write $\cC$ as a tensor product in the category ${}_K^K\cY\cD$ of algebras of the form $\LL[X,X^{-1}]$ (Laurent polynomials) and $\LL[X]/(X^n-\mu)$, $\mu\in \LL^\times$. This interpretation can be used to obtain an alternative proof of Proposition \ref{D_beta_mu}, since $\LL\langle X^\pm_1,\ld,X^\pm_m\rangle/J$ maps surjectively on such a tensor product.
\end{remark} 
 
The classification of graded-division algebras as $\cD(K,\beta_{ij},\mu_i)$ has a disadvantage of being dependent on the choice of decomposition \eqref{eq:cyclic_decomp} of $K$, so one may prefer to use the functions $\beta$ and $\mu$ rather than their values at chosen generators of $K$. Let us record the properties of $\mu$ with respect to multiplication. First of all, for any $g\in\Tor(T)$ and any divisor $n$ of $o(g)$, the element  $X_g^n$ is a scalar multiple of $X_{g^n}$ and $(X_g^n)^{o(g^n)}=X_g^{o(g)}$, so we have:
\begin{equation}\label{GDe0}
n\mid o(g)\quad\Rightarrow\quad \mu(g^n)=\mu(g)(\LL^\times)^{[o(g)/n]},
\end{equation}
which is the image in $\LL^\times/(\LL^\times)^{[o(g)/n]}$ of the element $\mu(g)\in\LL^\times/(\LL^\times)^{[o(g)]}$.
If $g,h\in\Tor(K)$ have coprime orders, then $o(gh)=o(g)o(h)$ and $X_g$ commutes with $X_h$ because $\beta(g,h)=1$. In this case, we have 
 \[ 
 (X_gX_h)^{o(gh)}=(X_g^{o(g)})^{o(h)}(X_h^{o(h)})^{o(g)}
 \]
and hence
\begin{equation}\label{GDe1}
\gcd(o(g),o(h))=1\quad\Rightarrow\quad\mu(gh)=\mu(g)^{o(h)}\mu(h)^{o(g)}.
\end{equation}
This expression is well-defined as an element of $\LL^\times/(\LL^\times)^{[o(gh)]}$.
 
Thus it suffices to describe the restriction of $\mu$ to $\Tor_p(K)$, for each prime divisor $p$ of $|\Tor(K)|$.  So, consider $g,h\in\Tor_p(K)$, with $o(g)=p^k$, $o(h)=p^\ell$ and, say, $k\ge\ell$. If $k=\ell$, assume that $o(gh)=p^k$. If $k>\ell$, this holds automatically. Using the formula 
\[
(X_gX_h)^n=\beta(g,h)^{-\binom{n}{2}}X_g^nX_h^n,
\] 
we obtain:  
\begin{align}
o(g)=p^k>o(h)=p^\ell & \quad\Rightarrow\quad  \mu(gh)=\mu(g)\mu(h)^{p^{k-\ell}};\label{GDe2}\\
o(g)=o(h)=o(gh)=p^k & \quad\Rightarrow\quad \mu(gh)=\begin{cases}
\mu(g)\mu(h) & \text{if $p$ is odd;}\\
\beta(g,h)^{p^{k-1}}\mu(g)\mu(h)&\text{if }p=2.
\end{cases}\label{GDe3}
\end{align}
Equation \eqref{GDe3} holds for $k\ge 1$; also note that $\beta(g,h)^{p^{k-1}}\in\{\pm 1\}$.
 
\begin{theorem}\label{th:D_beta_mu}
Let $\LL$ be a field and let $K$ be a finitely generated abelian group. For any alternating bicharacter $\beta:K\times K\to \LL^\times$ and any $\mu\in F_K\bydef\prod_{t\in\Tor(K)}F_t$, $F_t:=\LL^\times/(\LL^\times)^{[o(t)]}$, satisfying Equations \eqref{GDe0} through \eqref{GDe3}, there is a unique, up to isomorphism, graded-division algebra $\cD(K,\beta,\mu)=\bigoplus_{t\in K}\cD_t$ over $\LL$ such that 
\begin{enumerate}
\item[(i)] $\dim\cD_t=1$ for all $t\in K$; 
\item[(ii)] $x y=\beta(s,t)y x$ for all $x\in\cD_s$, $y\in\cD_t$, $s,t\in K$;
\item[(iii)] $x^{o(t)}\in\mu(t)$ for all $x\in\cD_t$, $t\in\Tor(K)$. 
\end{enumerate}
Every graded-division algebra $\cC=\bigoplus_{t\in K}\cC_t$ with $\dim\cC_t=1$, for all $t\in K$, is isomorphic to one, and only one, of the graded-division algebras $\cD(K,\beta,\mu)$.
\end{theorem}

\begin{proof}
Fix a decomposition \eqref{eq:cyclic_decomp} of $K$ such that the orders of those $a_i$ that belong to $\Tor(K)$ are prime powers. Let $\beta_{ij}=\beta(a_i,a_j)$, for all $i<j$, and let $\mu_i\in\LL^\times$ be any representative of $\mu(a_i)$, for all $i$ with $a_i\in\Tor(K)$. We have already seen that the graded-division algebra $\cD(K,\beta_{ij},\mu_i)$ of Proposition \ref{D_beta_mu} satisfies conditions (i) and (ii). Also, the function $\mu'\in F_K$ associated to $\cD(K,\beta_{ij},\mu_i)$ satisfies Equations \eqref{GDe0} through \eqref{GDe3} and coincides with $\mu$ on the generators $a_i$ of $\Tor(K)$. It is easy to see that this forces $\mu=\mu'$. The result follows.
\end{proof}

\begin{remark}\label{re:tensor_prod_decomp}
If $K$ is finite, we have a \emph{canonical} decomposition $K=\prod_p\Tor_p(K)$, which yields the decomposition of $\cC$ as an ordinary tensor product of the graded subalgebras $\cC(p)\bydef\cC_{\Tor_p(K)}$. 
\end{remark}

\section{The real case}\label{se:real}

We will now classify, up to isomorphism, the finite-dimensional graded-division algebras whose support is a (finite) abelian group over the field $\RR$, following the steps {\sc CP1}-{\sc CP3} proposed in Section \ref{se:gen}. In the treatment of algebras with $1$-dimensional homogeneous components, we will follow the approach of Section \ref{se:1d}. We keep the notation introduced in these two sections. Since we do not use the topology of $\RR$, our results are valid for any real closed field $\FF$ (if we interpret $\CC$ as $\FF[\sqrt{-1}]$, which is the algebraic closure of $\FF$). We fix $\bi=\sqrt{-1}$.

\medskip

{\sc CP1:} The only finite field extensions are $\RR$ and $\CC$. By a classical result of Frobenius, the only finite-dimensional central division algebras over $\RR$ are $\RR$ and $\HH$, and the only central division algebra over $\CC$ is $\CC$ itself.

\medskip

{\sc CP2:} We have to classify finite-dimensional graded-division algebras with $1$-dimensional components over $\LL=\RR$ and $\LL=\CC$.

\noindent $\boxed{\LL=\CC}$ Since $\LL$ is algebraically closed, in this case any finite-dimensional graded-division algebra $\cC$ has $1$-dimensional components. Also, $\mu$ is trivial, so such algebras are completely classified by $(K,\beta)$. 

\noindent $\boxed{\LL=\RR}$  Here $\LL$ is not algebraically closed, but $(\LL^\times)^{[n]}=\LL^\times$ for all odd $n$. Moreover, the only roots of unity in $\LL$ are $\pm 1$. By Remark \ref{re:tensor_prod_decomp}, we can write a finite-dimensional graded algebra $\cC$ with $1$-dimensional components as $\cC=\cC(2)\otimes\cC(2')$ where $\cC(2)=\cC_{\Tor_2(K)}$ and $\cC(2')=\cC_{\Tor_{2'}(K)}$. Since $\cC(2')$ has trivial $\beta$ and $\mu$, it must be the group algebra of $\Tor_{2'}(K)$ with coefficients in $\LL$. As to $\cC(2)$, it is determined by $(\Tor_2(K),\beta,\mu)$, where $\beta:\Tor_2(K)\times\Tor_2(K)\to\{\pm 1\}$, and $\mu$ can be regarded as a function $\Tor_2(K)\to\{\pm 1\}$ since $\LL/(\LL^\times)^{[n]}\simeq\{\pm 1\}$ for all even $n$. This function satisfies $\mu(e)=1$ and Equations \eqref{GDe2} and \eqref{GDe3} in Section \ref{se:1d}, which boil down to the following:
\begin{equation*}
\mu(gh)=\begin{cases}
\mu(g)&\text{if }o(g)>o(h);\\
\mu(g)\mu(h)&\text{if } o(gh)=o(g)=o(h)>2;\\
\beta(g,h)\mu(g)\mu(h)&\text{if }o(gh)=o(g)=o(h)=2.
\end{cases}
\end{equation*}
The last of the above equations clearly holds if one of the elements $gh, g, h$ equals $e$, so we have 
\begin{equation}\label{eq:quad_form}
\mu(gh)=\beta(g,h)\mu(g)\mu(h)\quad\text{for all }g,h\in K_{[2]}.
\end{equation}
In other words, regarding $K_{[2]}$ and $\{\pm 1\}$ as vector spaces over the field of two elements, the restriction of $\mu|_{K_{[2]}}$ is a \emph{quadratic form with polarization $\beta|_{K_{[2]}\times K_{[2]}}$} . Moreover, Equation \eqref{GDe0} yields $\mu(g)=\mu(g^{2^{k-1}})$ if $o(g)=2^k$ with $k>1$. Therefore, $\mu$ is completely determined by its restriction to $K_{[2]}$, which we still denote by $\mu$.

\begin{remark}\label{RCr1}
If $\Tor_2(K)$ is an elementary abelian $2$-group, then $\mu$ determines $\beta$.
\end{remark}

\medskip

{\sc CP3:} In the study of Galois descent for $\wt{\cD}=\End_\cE(\cV)$ as in Theorem \ref{GSt2}, the only nontrivial case is $\LL=\CC$; then $\cE=\cC$ and $\cC$ is determined by $(K,\beta)$, where $K$ has index $2$ in $T$ since $|\Gal(\CC/\RR)|=2$. By Remark \ref{GSr1}, $\cC$ is isomorphic to its complex conjugate, which is the only nontrivial Galois twist in this case, given by the generator $\gamma$ of $\Gal(\CC/\RR)$. It immediately follows from the commutation relations that this Galois twist of $\cC$ is determined by $(K,\gamma\circ\beta)=(K,\beta^{-1})$. Therefore, $\beta$ must take real values, so $\beta:K\times K\to\{\pm 1\}$. It follows that, if we scale the generators $X_{a_i}$ (associated to a decomposition \eqref{eq:cyclic_decomp}) so that $X_{a_i}^{o(a_i)}=1$, they will generate over $\RR$ a real form of the graded $\CC$-algebra $\cC$. 

Now, the Galois descent data for $\End_\cC(\cV)$ is just one $\CC$-antilinear automorphism $\psi=\psi_\gamma$ satisfying $\psi^2=\id$. Recall that $\psi$ is given by a pair $(\psi_0,\psi_1)$, where $\psi_0:\cC\to\cC$ is $\CC$-antilinear, and $\psi_1:\cV\to\cV$ is $\psi_0$-semilinear and must swap $\cV_\id$ and $\cV_\gamma$. Also recall that the pair $(\psi_0,\psi_1)$ is not unique, but can be adjusted using any nonzero homogeneous $c\in\cC$ to replace $\psi_0$ by $(\Int c)^{-1}\psi_0$ and $\psi_1$ by the map $v\mapsto\psi_1(v)c$. Thus, we may assume that $\psi_1$ is a ($\psi_0$-semilinear) isomorphism $\cV^{[t_0]}\to\cV$ where $t_0$ is an arbitrarily chosen element of the coset $T\smallsetminus K$ ($t_0=t_\gamma^{-1}$ in the notation of the proof of Theorem \ref{GSt2}). We fix $t_0$, so $\psi_1$ is determined up to multiplication by a nonzero element $c\in\cC_e=\CC$ and $\psi_0$ is uniquely determined.

\begin{lemma}\label{RCl1}
Any $\CC$-antilinear automorphism of the graded algebra $\cC$ is involutive.
\end{lemma}

\begin{proof}
As we have seen above, the graded algebra $\cC$ admits a real form. Using a homogeneous basis $X_s$, $s\in K$, of one of these forms, any $\CC$-antilinear automorphism of $\cC$ as a graded algebra is given by $X_s\mapsto\chi(s)X_s$, where $\chi\in \Hom_\ZZ(K,\CC^\times)$. Now the square of this automorphism sends $X_s$ to $\lvert \chi(s)\rvert^2X_s=X_s$, for any $s\in K$.
\end{proof}

It follows from Lemma \ref{RCl1} that the set of fixed points of $\psi_0$ is a real form of the graded algebra $\cC$, which we will denote by $\cC_\RR$, and we pick the elements $X_s$, $s\in K$, to be in $\cC_\RR$, so that $\cC_\RR=\bigoplus_{s\in K}\RR X_s$ and $\psi_0(X_s)=X_s$. Pick nonzero $v_1\in\cV_e$ and $v_2\in \cV_{t_0}$, so $\cV_\id=v_1\cC$ and $\cV_\gamma=v_2\cC$. Relative to the $\cC$-basis $\{v_1,v_2\}$, we can represent $\psi_1$ by the matrix $\begin{bmatrix}0&c_1\\c_2&0\end{bmatrix}$. Here $c_2\in\cC_e=\CC$ and $c_1\in\cC_{t_0^2}$. Adjusting $v_2$, we may assume that $c_2=1$. Then $\psi^2_1$ is $\cC$-linear ($\psi_0^2=\id_\cC$) and represented by $\begin{bmatrix}c_1&0\\0&\psi_0(c_1)\end{bmatrix}$. Since $\psi(r)=\psi_1r\psi_1^{-1}$, we have  
\[
\psi^2=\id\quad\Leftrightarrow\quad\psi_0(c_1)=c_1\in Z(\cC).
\]
This condition is equivalent to $c_1\in\RR X_{t_0^2}$ and $t_0^2\in\rad\beta$. Here 
\[
\rad\beta\bydef\{s\in K\mid\beta(s,t)=1\text{ for all }t\in K\}.
\] 
Since $\beta$ takes values $\pm 1$, we always have $K^{[2]}\subset\rad\beta$, so we conclude
\[
T^{[2]}\subset\rad\beta.
\]
Let us compute the fixed points of $\psi$ using matrix representation of $\End_\cC(\cV)$ relative to the $\cC$-basis $\{v_1,v_2\}$. We have
\[
X=\begin{bmatrix}x_{11}&x_{12}\\x_{21}&x_{22}\end{bmatrix}\quad\Rightarrow\quad
\psi(X)=\begin{bmatrix}0&c_1\\1&0\end{bmatrix}
\begin{bmatrix}\psi_0(x_{11})&\psi_0(x_{12})\\\psi_0(x_{21})&\psi_0(x_{22})\end{bmatrix}
\begin{bmatrix}0&c_1\\1&0\end{bmatrix}^{-1}.
\]
Hence, $X$ satisfies $\psi(X)=X$ if and only if $x_{22}=\psi_0(x_{11})$ and $x_{21}=c_1^{-1}\psi_0(x_{12})$ (compare with the proof of Theorem \ref{GSt2}: $c_{\gamma,\id}=1$ and $c_{\gamma,\gamma}=c_1$). Thus, writing $c_1=\lambda^{-1} X_{t_0^2}$, for some $\lambda\in\RR^\times$, we see that the real form $\cD$ of $\End_\cC(\cV)$ determined by our Galois descent has the following basis:
\begin{equation}\label{eq:rf}
Y_s\bydef\begin{bmatrix}X_s&0\\0&X_s\end{bmatrix},\quad JY_s,\quad
Y_{t_0s}\bydef\begin{bmatrix}0&X_{t_0^2}X_s\\\lambda X_s&0\end{bmatrix},\quad JY_{t_0s}\quad (s\in K),
\end{equation}
where $J\bydef\begin{bmatrix}\bi&0\\0&-\bi\end{bmatrix}$, $\cD_e=\langle I,J\rangle_\RR\simeq\CC$ and $\cD_t=\cD_e Y_t$ for all $t\in T$. 

\begin{remark}\label{rem:square_central}
Since $X_{t_0^2}\in Z(\cC)$, we have $Y_t^2=\lambda\begin{bmatrix}X_{t_0^2}X_{t_0^{-1}t}^2&0\\0&X_{t_0^2}X_{t_0^{-1}t}^2\end{bmatrix}\in Z(\cD)$ for all $t\in T\smallsetminus K$. Also, $((aI+bJ)Y_t)^2=(a^2+b^2)Y_t^2$ for all $a,b\in\RR$, hence $d^2\in Z(\cD)$ for all $d\in\cD_t$ with $t\in T\smallsetminus K$.
\end{remark}

We will now investigate the isomorphism problem for these real forms. A general fact is that the forms determined by $\psi$ and $\psi'$ are isomorphic if and only if there exists $\vp\in\Aut_\CC^T(\End_\cC(\cV))$ such that $\psi\vp=\vp\psi$ (indeed, $\vp$ maps the set of fixed points of $\psi$ onto that of $\psi'$). In terms of the corresponding pairs, $(\psi_0,\psi_1)$ for $\psi$, $(\psi_0',\psi_1')$ for $\psi'$ and $(\vp_0,\vp_1)$ for $\vp$ (where $\vp_0$ is $\CC$-linear), this condition is tantamount to the existence of a nonzero homogeneous $c\in\cC$ such that $\psi_0'\vp_0=(\Int c)^{-1}\vp_0\psi_0$ and $\psi_1'\vp_1(v)=\vp_1\psi_1(v)c$ for all $v\in\cV$. Since $\psi_1$ and $\psi_1'$ are chosen to have degree $t_0$, we get $c\in \cC_e=\CC$ and hence $\psi_0'\vp_0=\vp_0\psi_0$, which implies that the real forms $\cC_\RR$ and $\cC'_\RR$ determined by $\psi_0$ and $\psi_0'$, respectively, are isomorphic: $\vp_0$ maps $\cC_\RR$ onto $\cC'_\RR$. This reduces the isomorphism problem to the case where  $\cC_\RR=\cC_\RR'$ and $\psi_0=\psi_0'$. In this case $\vp_0$ restricts to an automorphism of the graded algebra $\cC_\RR$, so $\vp_0(X_s)=\chi(s) X_s$ for some $\chi\in\Hom_\ZZ(K,\RR^\times)$. 

On the one hand, we have $\psi_1^2(v)=vc_1$, for all $v\in \cV$, and similarly, $(\psi_1')^2(v)=vc_1'$. On the other hand, since $\psi_1'$ is $\psi_0$-linear, applying $\psi_1'$ to both sides of the identity  $\psi_1'(v)=(\vp_1\psi_1\vp_1^{-1})(v)c$, we obtain 
\[
(\psi_1')^2(v) =(\vp_1\psi_1^2\vp_1^{-1})(v)d,\quad\text{where }
d=c\psi_0(c)=|c|^2,
\]
where the last equality holds because $c\in\CC$ and $\psi_0$ is $\CC$-antilinear.
By comparing, we get $c_1'=\vp_0(c_1)d$. Since $c_1=\lambda^{-1}X_{t_0^2}$ and $c_1'=(\lambda')^{-1}X_{t_0^2}$, we get the following relation between the parameters $\lambda$ and $\lambda'$:
\[
\lambda/\lambda'=\chi(t^2_0)d,\quad\text{where } d>0.
\]
Therefore, we have a dichotomy: either $\chi(t_0^2)=1$ for all $\chi\in\Hom_\ZZ(K,\RR^\times)$ and then the sign of parameter $\lambda$ is uniquely determined, or there exists $\chi\in\Hom_\ZZ(K,\RR^\times)$ such that $\chi(t_0^2)=-1$ and then the sign of $\lambda$ can be changed to make $\lambda>0$. It is easy to see that the first alternative of the dichotomy occurs if and only if $t_0^2\in K^{[2]}$, which is equivalent to $K$ being a direct summand of $T$.

Finally, observe that taking $\vp_0=\id_\cC$ and $\vp_1$ represented by $\begin{bmatrix}1&0\\0&\sqrt{|\lambda|}\end{bmatrix}^{-1}$ yields an automorphism $\vp $ of the graded algebra $\End_\cC(\cV)$ that maps the real form spanned by the elements \eqref{eq:rf} to the one spanned by the elements of the same form, but with $\lambda$ replaced by $\sgn(\lambda)$.

To summarize our classification, it is convenient to introduce the following notation for some special cases of the graded-division algebras $\cD(K,\beta,\mu)$ as in Theorem \ref{th:D_beta_mu}, which can be presented by generators and relations as in Proposition \ref{D_beta_mu}.

\begin{definition}
Let $T$ be a finite abelian group. 
\begin{enumerate}
\item[(i)] Given an alternating bicharacter $\beta:T\times T\to\CC^\times$, let $\cD(T,\beta)=\bigoplus_{t\in T}\cD_t$ be the graded-division algebra over $\CC$ with $1$-dimensional components and support $T$ that is characterized, up to isomorphism, by the commutation relations $xy=\beta(s,t)yx$, for all $x\in\cD_s$, $y\in\cD_t$ and $s,t\in T$.
\item[(ii)] Given an alternating bicharacter $\beta:T\times T\to\{\pm 1\}$ and a quadratic form $\mu:T_{[2]}\to\{\pm 1\}$ with polarization $\beta|_{T_{[2]}\times T_{[2]}}$ (recall Equation \eqref{eq:quad_form}), let $\cD(T,\beta,\mu)=\bigoplus_{t\in T}\cD_t$ be the graded-division algebra over $\RR$ with $1$-dimensional components and support $T$ that is characterized, up to isomorphism, by the properties  $xy=\beta(s,t)yx$, for all $x\in\cD_s$, $y\in\cD_t$ and $s,t\in T$, and $x^{o(t)}\in\mu(t^{o(t)/2})\RR_{>0}$, for all nonzero $t\in T$ with even $o(t)$.
\end{enumerate} 
\end{definition}

We have proved the following result:

\begin{theorem}\label{RCt1}
Any finite-dimensional graded-division algebra over $\RR$ with abelian support $T$ is graded-isomorphic to one of the following:
\begin{enumerate}
\item $\cD(T,\beta,\mu)\simeq\cD(\Tor_2(T),\beta,\mu)\ot\RR \Tor_{2'}(T)$ where $\beta:T\times T\to\{\pm 1\}$ is an alternating bicharacter (which is automatically trivial on $\Tor_{2'}(T)$) and $\mu:T_{[2]}\to \{\pm 1\}$ is a quadratic form with polarization $\beta|_{T_{[2]}\times T_{[2]}}$;
\item $\cD(T,\beta,\mu)\ot_\RR\HH$ where $\cD(T,\beta,\mu)$ is as in item (1) and $\HH$ is the quaternion algebra with trivial grading;
\item The real form $\cD=\bigoplus_{t\in T}\cD_t$ of $M_2(\cC)$ with $2$-dimensional components $\cD_t=\langle Y_t,JY_t\rangle_\RR$ where $\cC$ is the complexification of $\cD(K,\beta,\mu)=\bigoplus_{s\in K}\RR X_s$ as in item (1) with $K\subset T$ being a subgroup of index $2$, $T^{[2]}\subset\rad\beta$, $J=\diag(\bi,-\bi)$, and the elements $Y_t$, $t\in T$, defined as follows:
\begin{enumerate}
	\item[(a)] $Y_s=\begin{bmatrix}X_s&0\\0&X_s\end{bmatrix}$ and 
	$Y_{t_0s}=\begin{bmatrix}0&X_s\\ \delta X_s&0\end{bmatrix}$ for all $s\in K$, if $K$ is a direct summand of $T$, where $\delta\in\{\pm 1\}$ and $t_0$ is an arbitrarily chosen element of order $2$ in $T\smallsetminus K$;
	\item[(b)] $Y_s=\begin{bmatrix}X_s&0\\0&X_s\end{bmatrix}$ and 
	$Y_{t_0s}=\begin{bmatrix}0&X_{t_0^2}X_s\\ X_t&0\end{bmatrix}$ for all $s\in K$, if $K$ is not a direct summand of $T$, where $t_0$ is an arbitrarily chosen element in $T\smallsetminus K$;
\end{enumerate} 
\item[(4)] $\cD(T,\beta)$ where $\beta:T\times T\to\CC^\times$ is an alternating bicharacter. 
\end{enumerate}
Moreover, two algebras in different items are not graded-isomorphic to one another, nor two algebras in the same item if they have different parameters $K$, $\beta$, $\mu$ and $\delta$ (as applicable), except that, in item (4), $\beta$ and $\beta^{-1}$ give graded-isomorphic algebras over $\RR$ (but not over $\CC$ unless $\beta$ takes values in $\{\pm 1\}$). Finally, the parameters $\mu$ and $\delta$ in item (3) depend on the choice of $t_0$, and the previous assertion applies only if the same choice of $t_0$ has been made for both algebras.\qed 
\end{theorem}

As stated above, the parameters $\mu$ and $\delta$ defining the same (graded-isomorphism class of) $\cD$ in item (3) depend on the choice of $t_0\in T\smallsetminus K$, so we will denote them $\mu_{t_0}$ and $\delta_{t_0}$. Let us investigate how they change if we replace $t_0$ by $t_0'=t_0g$, $g\in K$. 

Note that $\beta$ does not change, because it is an invariant of $\cC$ (which is itself an invariant of $\cD$, since $\cD\ot_\RR\CC\simeq\End_\cC(\cV)$ and also $\cC\simeq\cD_K$). On the other hand, the real form $\cC_\RR$ will change to $\cC'_\RR$. Indeed, $\psi_0$ is replaced by $(\Int X_g)^{-1}\psi_0$, so $\cC'_\RR$ is spanned by the elements $X'_h=X_h$, for all $h\in K$ with $\beta(g,h)=1$, and $X'_h=\bi X_h$, for all $h\in K$ with $\beta(g,h)=-1$. Therefore, we have 
\begin{equation}\label{eq:change_mu}
\mu_{t_0 g}(h)=\mu_{t_0}(h)\beta(g,h),\quad\text{for all }t_0\in T\smallsetminus K,\, g\in K,\text{ and }h\in K_{[2]}.
\end{equation}
If $K$ is a direct summand of $T$, we restricted the choice of $t_0$ to elements of order $2$, i.e., $t_0\in T_{[2]}\smallsetminus K_{[2]}$, so we consider only $g\in K_{[2]}$ in this case. From Remark \ref{rem:square_central} (with $\lambda=\delta_{t_0}$), it follows that 
\begin{equation}\label{eq:invariant_def_nu}
d^2\in\delta_{t_0}\mu_{t_0}(t_0^{-1}t)\RR_{>0},\quad\text{for all }0\ne d\in\cD_t\text{ and }t_0,t\in T_{[2]}\smallsetminus K_{[2]} . 
\end{equation}
In particular, $d^2\in\delta_{t_0}\RR_{>0}$ for all $0\ne d\in\cD_{t_0}$. Applying this to $t_0'=t_0 g$ and $\delta_{t_0'}$, and comparing with what \eqref{eq:invariant_def_nu} gives for $t=t_0'$, we conclude that
\begin{equation}\label{eq:change_delta}
\delta_{t_0 g}=\delta_{t_0}\mu_{t_0}(g),\quad\text{for all }t_0\in T_{[2]}\smallsetminus K_{[2]}\text{ and }g\in K_{[2]}.
\end{equation}
Using the transition formulas \eqref{eq:change_mu} and \eqref{eq:change_delta}, it is possible to compare graded-division algebras in item (3) of Theorem \ref{RCt1} even when they are presented using different elements $t_0$.

In order to state a parametrization that does not involve the choice of $t_0$, we introduce the following generalization of ``nice maps'' from \cite{R}:

\begin{definition}\label{df:nice}
Let $T$ be a finite abelian group, $K\subset T$ be a subgroup of index $2$, and $\beta:K\times K\to\{\pm 1\}$ be an alternating bicharacter. A map $\nu:T\smallsetminus K\to\{\pm 1\}$ will be called \emph{$(T,K,\beta)$-admissible} if 
\begin{equation}\label{eq:def_nice}
\frac{\nu(tgh)\nu(t)}{\nu(tg)\nu(th)}=\beta(g,h),\quad\text{for all }t\in T\smallsetminus K,\, g\in K,\text{ and }h\in K_{[2]}.
\end{equation}
(Of course, since the values are $\pm 1$, division above is the same as multiplication, but it makes the formula look more intuitive.)
We define the following equivalence relation on the set of $(T,K,\beta)$-admissible maps: $\nu\sim\nu'$ if and only if the map $t\mapsto\nu'(t)/\nu(t)$ is constant on each $K_{[2]}$-coset in $T\smallsetminus K$.
\end{definition}

Note that, for any $t\in T$, the map
\[
\nu_t: K_{[2]}\to\{\pm 1 \},\;h\mapsto\frac{\nu(th)}{\nu(t)}
\]
is the same for all $\nu$ within an equivalence class, and Equation \eqref{eq:def_nice} implies (by restricting $g$ to $K_{[2]}$) that $\nu_t$ satisfies Equation \eqref{eq:quad_form}, i.e., $\nu_t$ is a quadratic form with polarization $\beta|_{K_{[2]}\times K_{[2]}}$. Moreover, Equation \eqref{eq:def_nice} is equivalent to the following, for some, and hence all, $t\in T\smallsetminus K$:
\begin{equation}\label{eq:other_def_nice}
\nu_{tg}(h)=\nu_t(h)\beta(g,h),\quad\text{for all }g\in K\text{ and }h\in K_{[2]}.
\end{equation}
Now we can reformulate the parametrization of Theorem \ref{RCt1} (cf. \cite[Theorem 23]{R}):

\begin{corollary}
Let $G$ be an abelian group. The set of isomorphism classes of finite-dimensional $G$-graded algebras $\cD$ over $\RR$ that are graded-division algebras is in bijection with the disjoint union of the following four sets:

\noindent $\boxed{\cD_e=\RR}$ Triples $(T,\beta,\mu)$ where $T$ is a finite subgroup of $G$, $\beta:T\times T\to\{\pm 1\}$ is an alternating bicharacter, and $\mu:T_{[2]}\to\{\pm 1\}$ is a quadratic form with polarization $\beta|_{T_{[2]}\times T_{[2]}}$;

\noindent $\boxed{\cD_e=\HH}$ Triples $(T,\beta,\mu)$ as above;

\noindent $\boxed{\CC\simeq \cD_e\not\subset Z(\cD)}$ Quadruples $(T,K,\beta,\nu)$ where $K\subset T$ are finite subgroups of $G$ with $[T:K]=2$, $\beta:K\times K\to\{\pm 1\}$ is an alternating bicharacter with $T^{[2]}\subset\rad\beta$, and $\nu$ is as follows:
\begin{enumerate}
	\item[(a)] If $K$ is a direct summand on $T$, then $\nu$ is a $(T_{[2]},K_{[2]},\beta_{[2]})$-admissible map (see Definition \ref{df:nice}) where $\beta_{[2]}\bydef\beta|_{K_{[2]}\times K_{[2]}}$;
	\item[(b)] If $K$ is not a direct summand of $T$, then $\nu$ is an equivalence class of $(T,K,\beta)$-admissible maps;
\end{enumerate}

\noindent $\boxed{\CC\simeq \cD_e\subset Z(\cD)}$ Pairs $(T,\{\beta,\beta^{-1}\})$ where $T$ is a finite subgroup of $G$ and $\beta:T\times T\to\CC^\times$ is an alternating bicharacter.
\end{corollary}

\begin{proof}
Only the third case requires proof. We now explain how $\mu$ and $\delta$ (if applicable) are related to $\nu$. 

\noindent (a) Suppose $K$ is a direct summand of $T$. Given $\mu=\mu_{t_0}$ and $\delta=\delta_{t_0}$ corresponding to an element $t_0\in T_{[2]}\smallsetminus K_{[2]}$, we define $\nu$ as follows:
\[
\nu(t)\bydef\delta\mu(t_0^{-1}t),\quad\text{for all }t\in T_{[2]}\smallsetminus K_{[2]}.
\]
The fact that $\mu$ is a quadratic form with polarization $\beta_{[2]}$ immediately implies that $\nu$ is a $(T_{[2]},K_{[2]},\beta_{[2]})$-admissible map. Equations \eqref{eq:change_mu} and \eqref{eq:change_delta} can be used to verify that $\nu$ does not depend on the choice of $t_0$, but a better way is to observe that, in view of Equation \eqref{eq:invariant_def_nu}, we have $d^2\in\nu(t)\RR_{>0}$ for all $0\ne d\in\cD_t$. Conversely, given $\nu$ and $t_0$, we define $\mu\bydef\nu_{t_0}$ and $\delta\bydef\nu(t_0)$.

\noindent (b) Suppose $K$ is not a direct summand of $T$. Let $t_1,\ldots,t_\ell$ be representatives of the $K_{[2]}$-cosets contained in $T\smallsetminus K$, where $\ell=[K:K_{[2]}]$. Given a family $\mu_t$, $t\in T\smallsetminus K$, of quadratic forms with polarization $\beta_{[2]}$ and satisfying Equation \eqref{eq:change_mu}, pick $\delta_i\in\{\pm 1\}$ for $i=1,\ldots,\ell$ and define
\[
\nu(t)\bydef\delta_i\mu_{t_i}(t_i^{-1}t),\quad\text{for all }t\in t_i K_{[2]}.
\]
This defines a map $\nu:T\smallsetminus K\to\{\pm 1\}$. For any $t\in T\smallsetminus K$, we can write $t=t_i g$ for some $i$ and $g\in K_{[2]}$, hence
\[
\nu_t(h)=\frac{\nu(t_igh)}{\nu(t_ig)}=\frac{\delta_i\mu_{t_i}(gh)}{\delta_i\mu_{t_i}(g)}=\mu_{t_i}(h)\beta(g,h)=\mu_{t_i g}(h)=\mu_t(h),\quad\text{for all }h\in K_{[2]}.
\]
We have shown that $\nu_t=\mu_t$ for all $t\in T\smallsetminus K$, so Equation \eqref{eq:change_mu} becomes Equation \eqref{eq:other_def_nice}, which proves that $\nu$ is a $(T,K,\beta)$-admissible map. Clearly, by varying the $\delta_i$ we fill in an equivalence class of $(T,K,\beta)$-admissible maps. Conversely, given a $(T,K,\beta)$-admissible map $\nu$, we define $\mu_t\bydef\nu_t$, which is a family of quadratic forms with polarization $\beta_{[2]}$ and satisfying Equation \eqref{eq:change_mu}.
\end{proof}

We conclude this section by sketching an alternative approach to classifying graded-division algebras $\cD$ in item (3) of Theorem \ref{RCt1} (without Galois descent). We have $\cD_e\simeq\CC$ and $\cD_e$ is not central, so the support $K$ of $\cC\bydef\mathrm{Cent}_\cD(\cD_e)$ is a subgroup of index $2$ in $T$. Moreover, $\cC=\cD_K$. We have seen in Remark \ref{rem:square_central} that $d^2\in Z(\cD)$ for all $d\in\cD_t$ with $t\in T\smallsetminus K$, but this can be proved directly (see e.g. \cite[Remark 21]{R}). As before, fix $t_0\in T\smallsetminus K$ and pick nonzero $d_0\in\cD_{t_0}$. Then the automorphism $\Int d_0$ of the graded algebra $\cD$ is involutive, so it determines a $\ZZ_2$-grading $\cD=\cD_{\bar{0}}\oplus\cD_{\bar{1}}$, which is compatible with the $T$-grading. Therefore, we get a refined grading on $\cD$ by $T\times\ZZ_2$, whose homogeneous components 
\[
\cD_{(t,i)}\bydef\{d\in\cD_t\mid d_0 d d_0^{-1}=(-1)^i d\},\quad\text{for all }t\in T\text{ and } i\in\ZZ_2,
\]
are $1$-dimensional. Hence, the $T$-graded subalgebra $\cD_{\bar{0}}=\mathrm{Cent}_\cD(d_0)$ has $1$-dimen\-sional components and support $T$, i.e., it falls in item (1) of Theorem \ref{RCt1}. The graded-isomorphism class of $\cD_{\bar{0}}$ depends only on $t_0$, i.e., it does not depend on the choice of $0\ne d_0\in\cD_{t_0}$. Indeed, for any $0\ne c\in\cD_e\simeq\CC$, an isomorphism $\mathrm{Cent}_\cD(d_0)\to \mathrm{Cent}_\cD(cd_0)$ is given on homogeneous elements by $x\mapsto x$ if $\deg(x)\in K$ and $x\mapsto \frac{d}{|d|}x$ otherwise. The graded algebra $\cD$ can be recovered as $\cD_e\wh{\ot}\cD_{\bar{0}}$, where $\wh{\ot}$ denotes the tensor product in the category of $\RR$-superalgebras and we take the superalgebra structures $\cD_e=\RR\,1\oplus\RR\bi$ and $\cD_{\bar{0}}=(\cD_{\bar{0}})_K\oplus (\cD_{\bar{0}})_{T\smallsetminus K}$. The $T$-grading on the tensor product comes from the $T$-grading of $\cD_{\bar{0}}$ (since $\cD_e$ has a trivial $T$-grading). Note that the tensor product of superlagebras is again a superalgebra, but we forget about the superalgebra structure (which would be $(\cD_{\bar{0}})_K\oplus\bi(\cD_{\bar{0}})_{T\smallsetminus K}$ for the even part of $\cD$ and $(\cD_{\bar{0}})_{T\smallsetminus K}\oplus\bi(\cD_{\bar{0}})_{K}$ for the odd part).
The correspondence with our approach through Galois descent is the following: $\cC$ is the same (up to graded-isomorphism) and $\cC_\RR=(\cD_{\bar{0}})_K$.

\section{Gradings of fields}\label{se:fields}

\subsection{Graded field extensions}

\begin{definition}\label{dNSGF}
Let $\KK/\FF$ be a field extension and let $G$ be a group. We say that $\KK$ is a \emph{$G$-graded extension of $\FF$} if the additive group of $\KK$ is endowed with a direct sum decomposition
\begin{equation}\label{eGFE}
\KK=\bigoplus_{g\in G} \KK_g,
\end{equation}
where $\KK_g\ne 0$, $\KK_g\KK_h\subset \KK_{gh}$, for all $g,h\in G$, and $\KK_e=\FF$.
\end{definition}

In other words, the field $\KK$ is a $G$-graded $\FF$-algebra with $1$-dimensional homogeneous components and support $G$. We will show in a moment that this implies that $\KK/\FF$ is algebraic, but note that it need not be normal or separable. For example, $\QQ[\sqrt[3]{2}]/\QQ$ is $\ZZ_3$-graded and $\FF(\sqrt[p]{X})/\FF(X)$, with $\chr\FF=p$, is $\ZZ_p$-graded.

\begin{remark}
A finite Galois field extension $\KK/\FF$ can be considered as Hopf-Galois for the Hopf algebra $\FF^\Gamma$ of $\FF$-valued functions on $\Gamma=\Gal(\KK/\FF)$ (for which $\KK$ is a right comodule algebra) --- see e.g. \cite[Chapter 8]{Mont}. A field extension $\KK/\FF$ is a $G$-graded extension in the above sense if and only if it is Hopf-Galois for the group algebra $\FF G$ (for which $\KK$ is a comodule algebra by means of its $G$-grading).
\end{remark}

A more general object is a commutative graded-division algebra, which we will call a \textit{graded-field}. The support of a graded-field is an abelian subgroup of $G$, so we will assume that $G$ is equal to the support and, hence, is abelian. Also, the identity component is a field extension of $\FF$. In this section, we are interested in the following questions: 
\begin{enumerate}
	\item Which field extensions can be made graded field extensions?
	\item Which graded-fields are fields?
\end{enumerate}

The following result gives some restrictions on the groups that can be the support of a graded field extension. 

\begin{proposition}\label{psgf1n} 
Let $\KK/\FF$ be a $G$-graded field extension. Then $G$ is a torsion abelian group and there is a subgroup $M$ in the multiplicative group $\KK^\times$ such that $\FF^\times\subset M$ and $G\simeq M/\FF^\times$. In particular, 
	\begin{enumerate}
		\item[(i)] If $\KK$ is a finite field, then $G$ is a cyclic group;
		\item[(ii)] $\KK/\FF$ is an algebraic extension.
	\end{enumerate}
\end{proposition}

\begin{proof}
Assume, to the contrary, that there is an element $g\in G$ of infinite order and let $0\ne x\in\KK_g$. Then all powers $x^n$, $n\in\ZZ$, have different degrees and are therefore linearly independent over $\FF=\KK_e$. This means that $\KK_{\langle g\rangle}$ is isomorphic to the ring of Laurent polynomials over $\FF$. On the other hand, Lemma \ref{lm:inverse}, below, implies that $\KK_{\langle g\rangle}$ must be a subfield of $\KK$. This is a contradiction. 

Let $M=\KK^\times_\mathrm{gr}$, the group of nonzero homogeneous elements of $\KK$. Clearly, it is a subgroup of $\KK^\times$, and   the fact $M/\FF^\times\simeq G$ is just a special case of \eqref{eq:short_exact}. If $\KK$ is a finite field then $\KK^\times$ is cyclic, so $G$ must be cyclic. 

Let $x\in\KK_g$ for some $g\in G$. As we have shown, the order of $g$ is finite, so $x^{o(g)}$ is in $\FF$, which implies that $x$ is algebraic over $\FF$. Since $\KK$ is spanned by homogeneous elements, $\KK/\FF$ is an algebraic extension.  
\end{proof}

\begin{lemma}\label{lm:inverse}
Let $G$ be a group and let $\cA$ be a unital $G$-graded algebra. If $H$ is a subgroup of $G$ and an element $x\in\cA_H$ is invertible in $\cA$ then $x^{-1}\in\cA_H$. 
\end{lemma}

\begin{proof}
Let us partition the homogeneous components of $x^{-1}$ into left $H$-cosets: $x^{-1}=\sum_{i=1}^\ell y_i$ where $y_i\in\cA_{S_i}$ and $S_1,\ldots,S_\ell$ are distinct left $H$-cosets. Then $1=x^{-1}x=\sum_{i=1}^\ell y_i x$ and, since $x\in\cA_H$ and $SH=S$ for any left $H$-coset $S$, we see that $y_i\in S_i$ for all $i$. But $1\in\cA_e\subset\cA_H$, so we conclude that there exists $i$ such that $S_i=H$, and $y_j x=0$ for all $j\ne i$. This means $1=y_i x$ and hence $x^{-1}=y_i\in\cA_H$, as desired.
\end{proof}

\begin{proposition}\label{pSylow}
Let $\KK$ be a graded-field with support $G$ and let $\LL=\KK_e$. Then $\KK$ is a field if and only if $G$ is a torsion (abelian) group and, for every prime number $p$, the subalgebra $\KK(p)\bydef\KK_{\Tor_p(G)}$ is a field. 
\end{proposition}

\begin{proof}
The necessity of these conditions follows from Proposition \ref{psgf1n} and Lemma \ref{lm:inverse}. Let us prove the sufficiency. Let $x$ be a nonzero element of $\KK$. Since $x$ has finitely many nonzero components and every element of $G$ has a finite order, we see that $x\in\KK_H$ where $H$ is a finite subgroup of $G$. So, it is sufficient to prove that $\KK_H$ is a field. We can write $H$ as the direct product of its primary components (Sylow subgroups): $H=H_1\times\cdots\times H_m$ where $H_i\subset\Tor_{p_i}(G)$ and $p_1,\ldots,p_m$ are distinct primes. Then $\KK_H$ can be identified with the tensor product $\KK_1\ot_\LL\cdots\ot_\LL \KK_m$ where $\KK_i=\KK_{H_i}$. Since $\KK(p_i)$ is a field, Lemma \ref{lm:inverse} implies that $\KK_i$ is also a field. It is a finite extension of $\LL$ of degree $p_i^{k_i}$ for some $k_i$. Let us embed all $\KK_i$ over $\LL$ in the algebraic closure $\overline{\LL}$ and consider the composite $\wt{\KK}=\KK_1\cdots \KK_m$ in $\overline{\LL}$. Since $[\KK_i:\LL]$, $i=1,\ldots,m$ are pairwise coprime divisors of $[\wt{\KK}:\LL]$, we conclude that 
\[
\dim_\LL\wt{\KK}=[\wt{\KK}:\LL]\ge [\KK_1:\LL]\cdots[\KK_m:\LL]=\dim_\LL(\KK_1\ot_\LL\cdots\ot_\LL\KK_m)=\dim_\LL \KK_H.
\]
But then the surjective homomorphism of $\LL$-algebras $\KK_H=\KK_1\ot_\LL\cdots\ot_\LL \KK_m\to\wt{\KK}$, defined by $\alpha_1\ot\cdots\ot\alpha_m\mapsto\alpha_1\cdots\alpha_m$, must be an isomorphism, so $\KK_H$ is a field.
\end{proof}

From now on, we will assume that $G$ is a finite abelian group. We will also consider the identity component as the ground field $\FF$. The above proposition allows us to restrict ourselves to graded-fields which are graded by finite abelian $p$-groups, where $p$ is a prime number. Every such graded-field $\KK$ has the form
\begin{equation}\label{eGFp}
\KK\simeq \FF[X_1]/(X_1^{p^{k_1}}-\mu_1)\ot\cdots\ot  \FF[X_m]/(X_m^{p^{k_m}}-\mu_m),
\end{equation}
where $G=\langle a_1\rangle\times\cdots\times\langle a_m\rangle$, $o(a_i)=p^{k_i}$, and $\mu_i\in\FF^\times$ for all $i=1,\ld,m$. (This is a special case of the presentation in Proposition \ref{D_beta_mu}, with $\beta_{ij}=1$.)

In the proofs throughout this section, we will use the following criterion:

\begin{theorem}{\cite[Theorem 16 in Chapter 8]{L}}\label{tLang} 
Given a field $\FF$, a polynomial $X^n-\alpha\in\FF[X]$ is irreducible if and only if the following two conditions hold:
	\begin{enumerate}
		\item[(i)] If $p$ is a prime divisor of $n$ then $\alpha\not\in \FF^p$;
		\item[(ii)] If $4$ is a divisor of $n$ then $\alpha\not\in -4\:\FF^4$.\qed
	\end{enumerate}
\end{theorem}

In this section, the more traditional notation $(\FF^\times)^n$ will be used for the group $(\FF^\times)^{[n]}$. We start with the following necessary condition:

\begin{proposition}\label{pGFp}
Let $\KK$ be a graded-field of the form \eqref{eGFp} and let $U$ be the subgroup of $\FF^\times$ generated by $\mu_1,\ld,\mu_m$. If $\KK$ is a field then 
\[
|U(\FF^\times)^p/(\FF^\times)^p|=p^m.
\]
\end{proposition}

\begin{proof}
We proceed by induction on $m$. Denote by $U_k$ the subgroup of $\FF^\times$ generated by $\mu_1,\ldots,\mu_k$. If $m=1$ and $\FF[X_1]/(X_1^{p^{k_1}}-\mu_1)$ is a field then $X_1^{p^{k_1}}-\mu_1$ is irreducible, so $\mu_1\not\in\FF^p$. It follows that $|U_1(\FF^\times)^p/(\FF^\times)^p|=p$.
	
Now let $m\ge 2$. We set $H=\langle a_1\rangle\times\cdots\times\langle a_{m-1}\rangle$. Since $\KK$ is a field, $\KK_H$ is also a field by Lemma \ref{lm:inverse}. By induction hypothesis, $|U_{m-1}(\FF^\times)^p/(\FF^\times)^p|=p^{m-1}$. 
	
We have $\KK\simeq\KK_H[X_m]/(X_m^{p^{k_m}}-\mu_m)$. Since $\KK$ is a field, $X_m^{p^{k_m}}-\mu_m$ is an irreducible polynomial in $\KK_H[X_m]$, so $\mu_m\not\in (\KK_H)^p$. Assume $U_m(\FF^\times)^p/(\FF^\times)^p<p^m$, i.e., $\mu_m\in U_{m-1}(\FF^\times)^p$. Since in $\KK_H$ each $\mu_i$ is the $p^{k_i}$-th power of the image $x_i$ of $X_i$, for $i=1,\ld,m-1$, it follows that $\mu_m\in(\KK_H^\times)^p$, a contradiction. 
\end{proof}

Note that the condition in Proposition \ref{pGFp} is definitely not sufficient, because even in the case of $G\simeq\ZZ_4$, we have $-4\not\in(\QQ^\times)^2$, but the graded-field $\QQ[X]/(X^4+4)$ is not a field, since $X^4+4=(X^2+2X+2)(X^2-2X+2)$.
We have sufficiency in the following particular case:
\begin{equation}\label{eGF2}
\KK\simeq \FF[X_1]/(X_1^2-\mu_1)\ot\cdots\ot \FF[X_m]/(X_m^2-\mu_m).
\end{equation}

\begin{proposition}\label{pGF2}
Let $\KK$ be a graded-field of the form \emph{(\ref{eGF2})} with $\chr\FF\ne 2$, and let $U$ be the subgroup of $\FF^\times$ generated by $\mu_1,\ld,\mu_m$. Then $\KK$ is a field if and only if 
\[
|U(\FF^\times)^2/(\FF^\times)^2|=2^m.
\]
\end{proposition}

\begin{proof} 
Since $\mu_i\notin\FF^2$, each $\KK_i\bydef\FF[X_i]/(X_i^2-\mu_i)$ is a field. Let us embed them over $\FF$ in $\overline{\FF}$ and let $\wt{\KK}$ be the composite $\KK_1\cdots\KK_m$ in $\overline\FF$. For $\KK$ to be a field, it is necessary and sufficient that $[\wt{\KK}:\FF]=2^m$ (see the proof of Proposition \ref{pSylow}). But $\wt{\KK}/\FF$ is a Kummer extension of exponent $2$, so $[\wt{\KK}:\FF]=|\Lambda/(\FF^\times)^2|$ where $\Lambda=U(\FF^\times)^2$ (see Theorem \ref{tKE1}, below). The result follows.
\end{proof}

\begin{example}\leavevmode
\begin{enumerate}
\item[{(i)}] Over $\QQ$, there are $\ZZ_2^m$-graded field extensions for any $m=1,2,\ld$;
\item[{(ii)}] Over the field $\QQ_p$ of $p$-adic numbers, $p\ne 2$, $\ZZ_2^m$-graded field extensions exist only for $m=1,2$;
\item[{(iii)}] Over the field $\QQ_2$ of $2$-adic numbers, $\ZZ_2^m$-graded field extensions exist only for $m=1,2,3$.
\end{enumerate}
\end{example}

\begin{proof}
In case (i), the prime numbers $p_1=2, p_2=3,\ld$ are known to be independent mod $(\QQ^\times)^2$. Hence the field $\KK=\QQ(\sqrt{p_1},\sqrt{p_2},\ld,\sqrt{p_m})$ is graded by $\ZZ_2^m$. The grading is given by assigning to each element $p_1^{k_1/2}\cdots p_m^{k_m/2}$, $k_i\in\{0,1\}$, the degree $(\overline{k_1},\ld,\overline{k_m})\in\ZZ_2^m$. (These elements form a graded basis of $\KK$ over $\QQ$.)
	
Claims (ii) and (iii) are well-known facts about $\QQ_p^\times/(\QQ_p^\times)^2$.
\end{proof}

\subsection{Finite fields}\label{ssff}

Let $\KK$ be a finite field of characteristic $p$. Suppose $\KK$ is given a grading with support $G$ and let $\FF=\KK_e$. By Proposition \ref{psgf1n}, $G$ must be a cyclic group, and its order is $k=[\KK:\FF]$. Let $|\FF|=p^\ell$, so $|\KK|=p^{k\ell}$. Given $k$, we want to determine what $p$ and $\ell$ can appear, i.e., to characterize the extensions of finite fields that can be made $\ZZ_k$-graded extensions. By Proposition \ref{psgf1n}, a necessary condition is that $k$ be a divisor of $\dfrac{p^{k\ell}-1}{p^\ell-1}$. From the next result, it is clear that this condition is not sufficient (take e.g. $p=3$, $\ell=1$ and $k=4$).

\begin{theorem}\label{th:GF} 
Let $\KK=GF(p^n)$ and $\FF=GF(p^\ell)$ be Galois fields where $\ell$ divides $n$. Let $k=n/\ell$. Then there exists a $G$-grading on $\KK$ with support $G$ and the identity component $\FF$ if and only if the following two conditions hold:
	\begin{enumerate}
		\item[(i)] If $q$ is a prime divisor of $k$ then $q$ divides $p^{\ell}-1$;
		\item[(ii)] If $4$ divides $k$ then $4$ divides $p^{\ell}-1$.
	\end{enumerate}
If such a grading exists, the group $G$ must be cyclic of order $k$.
\end{theorem}

\begin{proof} 
The assertion about $G$ is already proved, so $G=\langle g\rangle$ where $o(g)=k$. Then $\KK$ must be isomorphic to $\FF[X]/(X^k-\mu)$ as a graded $\FF$-algebra, where $\mu\in\FF^\times$ and the grading is defined by setting $\deg X\bydef g$.  
Any value of $\mu$ yields a graded-field, and if it happens to be a field, then it must be isomorphic to $\KK$. Therefore, our question is to determine whether or not there exists $\mu\in\FF^\times$ such that $X^k-\mu$ is an irreducible polynomial in $\FF[X]$. 
	
According to Theorem \ref{tLang}, this happens if and only if $\mu\not\in\FF^q$, for any prime divisor $q$ of $k$, and also $\mu\not\in -4\FF^4$ if $4$ divides $k$. The multiplicative group $\FF^\times$ is cyclic of order $m=p^\ell-1$. If some prime divisor $q$ of $k$ is not a divisor of $m$ then $(\FF^\times)^q=\FF^\times$ and so the polynomial $X^k-\mu$ is never irreducible. If $4\mid k$ then $2$ is a prime divisor of $k$, so $m$ must be even and hence $p$ must be odd. Note that $4$ divides $m$ if and only if either $p\equiv 1\pmod{4}$, or $p\equiv 3\pmod{4}$ and $\ell$ is even. Assume that $4$ is not a divisor of $m$, i.e., $p\equiv 3\pmod{4}$ and $\ell$ is odd. Then $(\FF^\times)^4=(\FF^\times)^2$. Hence $-4(\FF^\times)^4=-4(\FF^\times)^2=-(\FF^\times)^2$. At the same time, $-1$ is not a square in $\FF$, so $\FF^\times=-(\FF^\times)^2\cup(\FF^\times)^2=(-4(\FF^\times)^4)\cup(\FF^\times)^2$. Therefore, $X^k-\mu$ is never irreducible in this case. We have proved the necessity of our conditions (i) and (ii).

To prove sufficiency, assume (i) and (ii) are satisfied. Then there exists $\mu\in\FF^\times$ such that $\mu\not\in\FF^q$ for any prime divisor $q$ of $k$. Indeed, by (i), every such $q$ is a divisor of $m$, so in the cyclic group $\FF^\times$ of order $m$ each $(\FF^\times)^q$ is a proper subgroup, and a finite cyclic group is not the union of its proper subgroups. If $4\nmid k$ then $X^k-\mu$ is irreducible for any $\mu$ as above. Now, if $4\mid k$ then $4\mid m$ by (ii), so either $p\equiv 1\pmod{4}$, or $p\equiv 3\pmod{4}$ and $\ell$ is even. It follows that $-1$ is a square in $\FF$, so $-4\FF^4\subset \FF^2$. Since $2$ is among the prime divisors of $k$, any $\mu$ as above satisfies $\mu\not\in -4\FF^4$ and, therefore, $X^k-\mu$ is irreducible.
\end{proof}

\begin{remark}
As we have seen in Section \ref{se:1d}, the isomorphism classes of graded-fields with support $\ZZ_k$ and identity component $\FF$ are parametrized by the cosets $\mu(\FF^\times)^k$ in $\FF^\times/(\FF^\times)^k$. The above proof shows that, in the case $\FF=GF(p^\ell)$, none of these graded-fields is a field unless conditions (i) and (ii) are satisfied, and if they are satisfied then the graded-field determined by $\mu(\FF^\times)^k$ is a field if and only if $\mu\notin(\FF^\times)^q$ or, equivalently, $\mu^{m/q}\ne 1$ for all prime divisors $q$ of $k$, where $m=p^\ell-1$. 
\end{remark}

\begin{example}\label{ex1} 
Let $\KK=GF(2^{q^\alpha})$ where $q$ is a prime number. Then any grading on $\KK$ is trivial.
\end{example}

\begin{proof} 
Using the notation of Theorem \ref{th:GF}, we must have $k=q^\beta$ and $\ell=q^{\alpha-\beta}$ for some $0\le\beta\le\alpha$. We claim that, for any $\beta\ne 0$, condition (i) is not satisfied. Indeed, $q$ is a prime divisor of $k$, but not of $m=2^\ell-1$, because $2^q\equiv 2\pmod{q}$ and hence $2^\ell\equiv 2\not\equiv 1\pmod{q}$.
\end{proof}

\begin{example}\label{ex3} 
Let $\KK=GF(p^{q\ell})$ where $p$ and $q$ are prime numbers and $q$ divides $p^\ell-1$. Then $\KK$ admits a $\ZZ_{q}$-grading with identity component $GF(p^\ell)$.
\end{example}

\begin{proof}
The conditions of Theorem \ref{th:GF} are satisfied, so the grading exists. Let us construct it explicitly. Let $\FF=GF(p^\ell)$ and consider the automorphism $\psi=\vp^\ell$ of $\KK$ where $\vp$ is the Frobenius automorphism: $\vp(x)=x^p$. The order of $\psi$ is $q$ and the set of its fixed points equals $\FF$. Since $q$ divides $p^\ell-1$, there is a subgroup of order $q$ in $\FF^\times$, i.e., $\FF$ contains a primitive $q$-th root of unity, say, $\zeta$ (so $\KK/\FF$ is a Kummer extension, see below). It follows that we have a $\ZZ_q$-grading $\KK=\bigoplus_{i\in\ZZ_q}\KK_i$ whose homogeneous components are the eigenspaces of $\psi$: $\KK_i=\{x\mid x^{p^\ell}=\zeta^i x\}$, for all $i\in\ZZ_q$.
\end{proof}

\subsection{Kummer extensions}\label{ssKE}

A finite Galois extension $\KK/\FF$ is called a \textit{Kummer extension of exponent $n$} if the Galois group $\Gamma =\Gal(\KK/\FF)$ is abelian of exponent dividing $n$ and $\FF$ contains a primitive $n$-th root of unity. The classification of such extensions is as follows. 

\begin{theorem}{\cite[Theorems 13 and 14 in Chapter 8]{L}}\label{tKE1} 
Let $\FF$ be a field containing a primitive $n$-th root of unity. Then we have the following:
\begin{enumerate}
	\item Given any subgroup $\Lambda$ of $\FF^\times$ such that $(\FF^\times)^n\subset\Lambda$ and $\Lambda/(\FF^\times)^n$ is finite, we have a canonical isomorphism from $\Lambda/(\FF^\times)^n$ onto the dual group $\wh{\Gamma} = \Hom(\Gamma,\FF^\times)$ of $\Gamma = \Gal(\FF(\Lambda^{1/n})/\FF)$, given by $a(\FF^\times)^n\to\chi_a$ where, for any $a\in\Lambda$, $\chi_a(\sigma)\bydef\sigma(\alpha)/\alpha$, where $\alpha$ is an element of $\FF(\Lambda^{1/n})$ such that $\alpha^n = a$. In particular, we have $[\FF(\Lambda^{1/n}) : \FF] = |\Lambda/(\FF^\times)^n|$.
	\item The mapping $\Lambda\mapsto \FF(\Lambda^{1/n})$ sets up an inclusion-preserving bijection from the set of subgroups $\Lambda$ of $\FF^\times$ such that $(\FF^\times)^n\subset\Lambda$ and $\Lambda/(\FF^\times)^n$ is finite, onto the set of isomorphism classes of Kummer extensions of $\FF$ of exponent $n$.\qed
\end{enumerate}
\end{theorem}

\begin{corollary}\label{tKE2}
Every Kummer extension is a graded extension in a canonical way. More precisely, let $\KK/\FF$ be a Kummer extension of exponent $n$, so $\KK=\FF(\Lambda^{1/n})$ and $\Gamma=\Gal(\FF(\Lambda^{1/n})/\FF)$ as in Theorem \ref{tKE1}. Then $\KK/\FF$ is a $G$-graded extension, with $G=\Lambda/(\FF^\times)^n\simeq\wh{\Gamma}$, as follows: the homogeneous component labeled by the coset represented by $a\in\Lambda$, is given by 
\begin{equation}\label{eq:grad_Kummer}
\KK_{a(\FF^\times)^n}=\FF\alpha\quad\text{where }\alpha^n=a
\end{equation}
and does not depend on the choice of the representative $a$ and $\alpha\in\KK$ such that $\alpha^n=a$.  
\end{corollary}

\begin{proof}
Since $\FF$ contains enough roots of unity, the $\Gamma$-action on $\KK$ diagonalizes, and the eigenspace decomposition is a grading by $\wh{\Gamma}$. The identity component of this grading is the set of fixed points of $\Gamma$, i.e., $\FF$, which implies that $\KK/\FF$ is a $\wh{\Gamma}$-graded extension.

For each character $\chi\in\wh{\Gamma}$, there is a unique coset $a(\FF^\times)^n\in\Lambda/(\FF^\times)^n$ such that $\chi=\chi_a$, i.e., $\chi_a(\sigma)=\sigma(\alpha)/\alpha$ where $\alpha\in\KK$ is any element satisfying $\alpha^n=a$. An element $x\in\KK$ is homogeneous of degree $\chi_a$ if and only if, for any $\sigma\in\Gamma$, we have 
\[
\sigma(x)=\chi_a(\sigma)x=(\sigma(\alpha)/\alpha)x.
\]
Now, $\alpha$ itself satisfies the above equation. Since the components of the grading are $1$-dimensional, $\KK_{a(\FF^\times)^n}=\KK_{\chi_a}=\FF\alpha$, as claimed.
\end{proof}

\begin{example}\label{ex2} 
Let $\KK=\QQ(\zeta_n)$ be the $n$-th cyclotomic extension of $\QQ$. The Galois group of $\Gamma=\Gal(\KK/\QQ)$ is isomorphic to $\ZZ_n^\times$. If $n=2^{k}p_1^{k_1}\cdots p_m^{k_m}$ where $2<p_1<\ldots<p_m$ are prime, $k\ge 0$ and $k_i>0$, then $\Gamma$ is the direct product of either $m$ (if $k\le 1$), $m+1$ (if $k=2$), or $m+2$ (if $k\ge 3$) cyclic subgroups of even order. Thus, $\Gamma_{[2]}$ is either $\ZZ_2^m$, $\ZZ_2^{m+1}$ or $\ZZ_2^{m+2}$. Since abelian groups of exponent $2$ are isomorphic to the groups of their rational characters, $\KK$ admits a grading by $\ZZ_2^m$, $\ZZ_2^{m+1}$ or $\ZZ_2^{m+2}$ (as indicated above). The identity component of the grading is the set of fixed points under the action of $\Gamma_{[2]}$. The only cases where this group is the whole $\Gamma$ are $n=1,2,3,4,6,8,12,24$. In these cases $\KK/\QQ$ is a graded field extension, and the grading is given by Equation \eqref{eq:grad_Kummer} above. 

For example, $\KK=\QQ(\zeta_8)$ can be written as $\KK=\QQ(\sqrt{2},\bi)=\QQ(\Lambda^{1/2})$ where $\Lambda=\langle 2,-1\rangle(\QQ^\times)^2$. We have $\Gamma=\langle\sigma\rangle\times\langle\tau\rangle$ where $\sigma(\sqrt{2})=-\sqrt{2}$, $\sigma(\bi)=\bi$, and $\tau$ is the complex conjugation restricted to $\KK$. The dual group $\wh{\Gamma}$ is generated by $\chi_2$ and $\chi_{-1}$, which are given by $\chi_2(\sigma)=-1$, $\chi_2(\tau)=1$, and $\chi_{-1}(\sigma)=1$, $\chi_{-1}(\tau)=-1$. The canonical $\wh{\Gamma}$-grading is the following ($\wh{\Gamma}\simeq\ZZ_2^2$):
\begin{equation}\label{eq:can_gr_z8}
\QQ(\zeta_8)=\QQ\oplus\QQ\sqrt{2}\oplus\QQ\bi\oplus\QQ\sqrt{2}\,\bi.
\end{equation}
\end{example}

The converse of Corollary \ref{tKE2} is false. Indeed, we have already seen that a graded field extension need not be Galois. The following is a partial converse:

\begin{proposition}\label{tKE3}
Let $\KK/\FF$ be a $G$-graded finite field extension. Assume that $\FF$ contains a primitive root of unity of degree $n=\exp(G)$. Then $\KK/\FF$ is a Kummer extension with $\Gal(\KK/\FF)\simeq\wh{G}$.
\end{proposition}

\begin{proof}
By our assumption, $G$ and $\wh{G}$ are isomorphic (non-canonically), so they have the same exponent $n$. The $G$-grading on $\KK$ yields a $\wh{G}$-action, $\wh{G}\to\Aut_\FF(\KK)$, given by $\chi\cdot x=\chi(g)x$ for all $x\in\KK_g$, $g\in G$ and $\chi\in\wh{G}$. Since the support of the grading is $G$, this action is faithful. Since $\KK_e=\FF$ and $\wh{G}$ separates points in $G$, the set of fixed points under the $\wh{G}$-action is $\FF$. Hence $\KK/\FF$ is Galois with $\Gal(\KK/\FF)\simeq\wh{G}$. 
\end{proof}

\begin{remark}
Both Corollary \ref{tKE2} and Proposition \ref{tKE3} are consequences of the fact that an action (resp., grading) by a finite group $G$ is the same as the action (resp., coaction) by the Hopf algebra $\FF G$, so it canonically corresponds to a coaction (resp., action) of the Hopf dual $(\FF G)^*=\FF^G$, and we have $\FF^G=\FF\wh{G}$ if $G$ is abelian and $\FF$ contains a primitive root of unity of degree $\exp(G)$ --- see e.g. \cite{Mont}.
\end{remark}

\begin{example}\label{ex4} 
Recall $\KK=\QQ(\zeta_8)$ from Example \ref{ex2}, with the canonical $\ZZ_2^2$-grading given by \eqref{eq:can_gr_z8}. We have $\KK\simeq \QQ[X]/(X^4+1)$, since the $8$-th cyclotomic polynomial is $X^4+1$, so $\KK/\QQ$ is a $\ZZ_4$-graded extension: 
\[
\QQ(\zeta_8)=\QQ\oplus\QQ\zeta_8\oplus\QQ\bi\oplus\QQ\zeta_8\bi.
\]
\end{example}

This shows that a Kummer extension may be a graded extension in a way different from the canonical one.

\end{document}